\documentclass[12pt]{amsart}
\usepackage[utf8]{inputenc}
\usepackage{indentfirst}
\usepackage{ragged2e}
\usepackage{times}
\usepackage{cancel}
\usepackage{amsmath,amsfonts,amssymb,amsthm}
\usepackage{bbding}
\usepackage{mathptmx}
\usepackage{setspace}
\usepackage[usenames]{color}
\usepackage{mathrsfs}
\usepackage{stmaryrd}
\usepackage{amscd}
\usepackage{cite}
\usepackage{cases}
\usepackage[all]{xy}           
\usepackage{bbm}
\usepackage{txfonts}
\usepackage{amscd}
\usepackage{tikz}
\usetikzlibrary{matrix}
\usepackage[shortlabels]{enumitem}

\usepackage{ifpdf}
\ifpdf
\usepackage[colorlinks=true,linkcolor=blue,filecolor=green,citecolor=green,final,backref=page,hyperindex]{hyperref}
\else
\usepackage[colorlinks=true,linkcolor=red,filecolor=green,citecolor=green,final,backref=page,hyperindex]{hyperref}
\fi
\usepackage{tikz}
\usepackage[active]{srcltx}

\makeatletter
\newcommand {\emptycomment}[1]{}

\newtheorem{thm}{Theorem}[section]
\newtheorem{lem}[thm]{Lemma}

\newtheorem{pro}[thm]{Proposition}

\newtheorem{rmk}[thm]{Remark}
\newtheorem{defi}[thm]{Definition}

\numberwithin{equation}{section}

\numberwithin{equation}{section}

\allowdisplaybreaks[4]
\begin{document}

\title[]{On simple restricted modules of Hamiltonian superalgebras with $p$-characters of height 0}

\author{Jingyi Zhang}
\address{School of Mathematics and Statistics, Northeast Normal University, Changchun 130024, Jilin, China}
\email{jyzhang831@nenu.edu.cn}

\author{Liangyun Chen*}
\address{School of Mathematics and Statistics, Northeast Normal University, Changchun 130024, China}
\email{chenly640@nenu.edu.cn}

\begin{abstract}
	Let $H(2,1;\underline{t})$ be Hamiltonian superalgebras over $\mathbb{F}$, an algebraically closed field of prime characteristic $p>3$, which are non-restricted simple Lie superalgebras, generally.
	In this paper, we study generalized $\chi$-reduced simple modules over $H(2,1;\underline{t})$. We proved that all generalized $\chi$-reduced Kac modules of $H(2,1;\underline{t})$ are simple with $p$-characters $\chi$ of height 0. Additionally, the isomorphism classes of these simple $H(2,1;\underline{t})$-modules are classified and their dimensions are determined.
	
\end{abstract}

\thanks{*Corresponding author.}
\thanks{This work is supported by NNSF of
China (No. 12271085), NSF of Jilin Province (No. YDZJ202201ZYTS589) and the Fundamental Research Funds for the Central Universities.}
\maketitle

\tableofcontents

\allowdisplaybreaks
	\section{Introduction}
For Lie superalgebras over fields of characteristic $0$ (non-modular Lie superalgebras), the classification of finite-dimensional simple Lie superalgebras has been completed\cite{Kac1,Kac2}. However, for Lie superalgebras over fields of characteristic
$p$ (modular Lie superalgebras), the classification of simple modular Lie superalgebras remains an open problem due to the restrictions imposed by the characteristic of the base field\cite{DAL,SBPGDAL}. In the study of modules of modular Lie superalgebras, the research on the modules of Cartan-type simple modular Lie superalgebras plays a crucial role in the classification efforts.

As is well known, eight families of finite-dimensional graded simple modular Lie superalgebras of Cartan-type have been constructed in \cite{YZZWDL,JYFQCZCPJ,WDLYHH,WDLJXY,WDLYZZXLW}.
Hamiltonian superalgebras are one of them. When Hamiltonian superalgebras are restricted Lie superalgebras, both their restricted and non-restricted modules have been studied\cite{YFY,YFYBS}.
	However, when Hamiltonian superalgebras are non-restricted Lie superalgebras, their representation theory remains unresolved.
	In \cite{YZZQCZ}, the concept of generalized restricted Lie algebras was extended to generalized restricted Lie superalgebras.
	In \cite{YFYBSYYL,YFY0}, Yao et al. studied simple modules of Zassenhaus superalgebras (a type of generalized restricted Lie superalgebras) with $p$-characters of height $0$ and $-1$, respectively. It's natural to think that we can study simple modules of Hamiltonian superalgebras through their generalized restricted structure.

	In \cite{WDLJXY2}, Liu and Yuan studied the automorphism group of modular graded Lie superalgebras of Cartan-type. Generalized $\chi$-reduced modules (see Definition \ref{Def2.8}) of
	Hamiltonian superalgebras are entirely determined by the orbit of the automorphism group of $p$-characters $\chi$.
In this paper, we study simple modules with $p$-characters of height $0$ over Hamiltonian superalgebras $H(2,1;\underline{t})$, further enriching the classification theory of modules for Hamiltonian superalgebras.
	The paper is organized as follows. In Section \ref{Preliminaries}, some preparations such as	the generalized restricted structure of Hamiltonian superalgebras are given. In Section \ref{3}, simple modules of $H(2,1;\underline{t})$ with $p$-characters of height 0 are studied. In Section \ref{4}, the isomorphism classes of these simple $H(2,1;\underline{t})$-modules are classified and their dimensions are determined.

\section{Preliminaries}\label{Preliminaries}
In this paper, we always assume that the base field $\mathbb{F} $ is algebraically closed and of characteristic $p>3$. $\mathbb{F}_p$ is an algebraically closed field of elements considered modulo $p$. All vector spaces (modules) $V$ are over $\mathbb{F} $ and $\mathbb{Z}_{2} $-graded, i.e., $V=V_{\bar{0} }\oplus V_{\bar{1}}$. For each $\mathbb{Z}_{2} $-homogeneous element $v\in V$, we denote by $\bar{v}$ the parity of $v$. Let $A=A_{\bar{0} }\oplus A_{\bar{1}}$ be a superalgebra. A homogeneous super-derivation $\mathcal{D}  $ of $A$ is a homogeneous linear transformation of $A$ such that $\mathcal{D}( ab)=\mathcal{D} (a ) b+( -1)^{\bar{\mathcal{D} } \bar{a}  } a\mathcal{D} (b )  $ for all $a,b\in A_{\bar {0}}\cup A_{\bar {1}}$.

	\subsection{Generalized restricted structure of Hamiltonian superalgebras}
\	

Set $\underline{t}=(t_1,t_2)\in \mathbb{N}^{2}$ and $\mathcal{O}(m;\underline{t})$ be the divided power algebra of two variables $x_1$ and $x_2$ with $\mathbb{F}$-basis $\{x_1^{(i_1)}x_2^{(i_2)}\mid 0\le i_1\le p^{t_1}-1,0\le i_2\le p^{t_2}-1 \}$ under the formul $$x_k^{(j_1)}x_k^{(j_2)}=C_{j_1+j_2}^{j_1}x_k^{(j_1+j_2)} ,\,\,\forall 0\le j_1,j_2\le p^{t_k}-1, k=1,2.$$ Set $\Lambda(1)$ be the Grassmann superalgebra of one variable $\xi$. Set $\Lambda(2,1;\underline{t})=\mathcal{O}(2;\underline{t} )\otimes _{\mathbb{F} } \Lambda(1)$. Then $\Lambda(2,1;\underline{t})$ has an $\mathbb{F}$-basis $$ \{ x_1^{ ( i_1  ) } x_2^{(i_2)}\xi^{j}\mid 0\le i_1 \le p^{t_1}-1, 0\le i_2 \le p^{t_2}-1, j=0,1 \} $$ and is $\mathbb{Z}_{2} $-graded with $$\Lambda(2,1;\underline{t})_{\bar{0} }={\rm span}_{\mathbb{F} } \{ x_1^{(i_1)}x_2^{ ( i_2  ) } \mid 0\le i_1 \le p^{t_1}-1, 0\le i_2 \le p^{t_2}-1 \} $$ and $$\Lambda(2,1;\underline{t})_{\bar{1} }={\rm span}_{\mathbb{F} } \{ x_1^{(i_1)}x_2^{ ( i_2 ) }\xi \mid 0\le i_1 \le p^{t_1}-1, 0\le i_2 \le p^{t_2}-1 \}. $$ Moreover, $\Lambda(2,1;\underline{t})$ are associative superalgebras in which the multiplication is given by
\begin{align*}
x_1^{ ( i_1  ) }x_2^{(i_2)}\xi^{j}\cdot  x_1^{ ( s_1  ) }x_2^{(s_2)}\xi^{l}=C_{i_1+s_1}^{i_1} C_{i_2+s_2}^{i_2}x_1^{ ( i_1+s_1  ) }x_2^{ ( i_2+s_2  ) }\xi^{j+l}, \\
0\le i_1,s_1\le p^{t_1}-1, 0\le i_2,s_2\le p^{t_2}-1, j,l\in\{0,1\}
\end{align*}
with the convention that
\begin{align*}
x_1^{(k_1)}x_2^{(k_2)}\xi^{l}=0, \text{ for } k_1 \ge p^{t_1} \text{ or }k_2\ge p^{t_2} \text{ or } l>1.
\end{align*}
$D_1, D_2, D_3$ are special super-derivations of $\Lambda(2,1;\underline{t})$ satisfying
$$\mathcal{D}_k(x^{(i)}\xi ^j)=x^{(i-1)}\mathcal{D}_k(x)\xi^j+jx^{(i)}\mathcal{D}_k(\xi)$$
for $0 \le i \le p^{t_1+t_2}-1$, $j = 0, 1$ and $k=1,2,3$.

Set
$$W(2,1;\underline{t})=\{\sum_{i=1}^3f_iD_i\mid f_i\in \Lambda(2,1;\underline{t})\}$$
and
\begin{align*}i^{\prime}=
\left\{\begin{matrix}
2, & i=1, \\
1, & i=2,\\
3, & i=3,
\end{matrix}\right.	
\qquad
\sigma(i)=
\left\{\begin{matrix}
1, & i=1, \\
-1, & i=2,\\
1, & i=3,
\end{matrix}\right.
\qquad
\tau(i)=
\left\{\begin{matrix}
\overline{0}, & i=1,2, \\
\overline{1}, & i=3.
\end{matrix}\right.
\end{align*}
Define a linear mapping $D_H:\Lambda(2,1;\underline{t})\longrightarrow W(2,1;\underline{t})$ so that
$$D_H(f)=\sum_{i=1}^3f_iD_i, \quad \forall f \in \Lambda(2,1;\underline{t}),$$
where
$$f_i=\sigma(i^\prime)(-1)^{\tau(i^\prime)d(f)}D_{i^\prime}(f),\quad i=1,2,3.$$
By direct calculation, we can know
$$\overline{H}(2,1;\underline{t})={\rm span}_\mathbb{F}\{D_H(f)\mid f\in \Lambda(2,1;\underline{t})\}.$$
$D_1$, $D_2$ and $D_3$ are special super-derivations of $\Lambda(2,1;\underline{t})$.
An examination of the commutative subalgebra of $\overline{H}(2,1;\underline{t})$ shows that
$$D_H(f)=\sum_{i=1}^3\sigma(i)(-1)^{\tau(i)d(f)}D_i(f)D_{i^\prime}.$$
It follows from Lemma 2.10 in \cite{YZZWDL} that
\begin{align*}
[D_H(f),D_H(g)]=&D_H(\sum_{i=1}^3\sigma(i)(-1)^{\tau(i)d(f)}D_i(f)D_{i^\prime}(g))\\
=&D_H(D_H(f)(g)),\quad \forall f,g\in \Lambda(2,1;\underline{t}).
\end{align*}
A straightforward calculation implies that
\begin{align*}
H(2,1;\underline{t})=&{\rm span}_\mathbb{F}\{D_H(f)\mid f\in \mathop{\oplus}\limits_{i=0}^{\xi -1}\Lambda(2,1;\underline{t})_i\}\\
=&{\rm span}_\mathbb{F}\{D_H({x_1^{(i_1)}x_2^{(i_2)}}\xi^j)\mid 0\le i_1\le p^{t_1}-1,\, 0\le i_2\le p^{t_2}-1,\,j=0,1 ,\,i_1+i_2+j\ge 1\}\\
=&{\rm span}_\mathbb{F}\{x_1^{(i_1-1)}x_2^{(i_2)}\xi^jD_2-x_1^{(i_1)}x_2^{(i_2-1)}\xi^jD_1-jx_1^{(i_1)}x_2^{(i_2)}D_3\\
&\mid 0\le i_1\le p^{t_1}-1,\, 0\le i_2\le p^{t_2}-1,\,j=0,1,\, i_1+i_2+j\ge 1\}\\
=&H_{\bar{0}}\oplus H_{\bar{1}},
\end{align*}
where
\begin{align*}
H_{\bar{0}}=&{\rm span}_\mathbb{F}\{D_H(x_1^{(i_1)}x_2^{(i_2)})\mid 0\le i_1\le p^{t_1}-1, \,0\le i_2\le p^{t_2}-1,\, i_1+i_2\ge 1\}
\\=&{\rm span}_\mathbb{F}\{x_1^{(i_1-1)}x_2^{(i_2)}D_2-x_1^{(i_1)}x_2^{(i_2-1)}D_1\mid 0\le i_1\le p^{t_1}-1, \,0\le i_2\le p^{t_2}-1,\,i_1+i_2\ge 1\}
\end{align*}
and
\begin{align*}
H_{\bar{1}}=&{\rm span}_\mathbb{F}\{D_H(x_1^{(i_1)}x_2^{(i_2)}\xi)\mid 0\le i_1\le p^{t_1}-1,\, 0\le i_2\le p^{t_2}-1\}
\\=&{\rm span}_\mathbb{F}\{x_1^{(i_1-1)}x_2^{(i_2)}\xi D_2-x_1^{(i_1)}x_2^{(i_2-1)}\xi D_1-x_1^{(i_1)}x_2^{(i_2)}D_3\mid 0\le i_1\le p^{t_1}-1,\, 0\le i_2\le p^{t_2}-1\}.
\end{align*}
It is a routine to check that $H(2,1;\underline{t})$ are Lie subsuperalgebras of the derivation superalgebras of $\Lambda(2,1;\underline{t})$ under the usual Lie bracket (\cite{YZZWDL}). Moreover, they are simple Lie superalgebras
and referred to as Hamiltonian superalgebras, the even part of which contains a subalgebra isomorphic to the usual Hamiltonian algebras(\cite{HS}).

The following notion of a generalized restricted Lie superalgebra generalizes the notion of a restricted Lie superalgebra.
\begin{defi} {\rm(\cite{YZZQCZ})}\label{2.2}
	Let $L=L_{\bar{0}} \oplus L_{\bar{1}}$ be a Lie superalgebra,
	if the even part  $L_{\bar{0}}$ is a generalized restricted Lie algebra and the odd part  $L_{\bar{1}}$ as an adjoint $L_{\bar{0}}$-module is a generalized restricted module,i.e., there exists an ordered basis $ E=\{e_{i} \mid i \in I\}$ of $L_{\bar{0}} $, an non-negative $|I| $-tuple $ \mathbf{s}=(s_{i})_{i \in I}$  and a generalized restricted map  $\varphi_{\mathrm{s}}: E \rightarrow L_{\bar{0}}$ sending  $e_{i}$ to $e_{i}^{\varphi_{\mathbf{s}}}$ such that  $(\operatorname{ad}e_{i})^{p^{s_i}}(y)=(\operatorname{ad}e_{i}^{\varphi_{\mathbf{s}}})(y)$ for all $y\in L$, then $L$ is a generalized restricted Lie superalgebra (GR Lie superalgebra) associated with $E$ and $\varphi_{\mathbf{s}}$.
\end{defi}

\begin{rmk}	
	
	$( \mathrm {i} ) $ One usually denotes the GR Lie superalgebra mentioned in \rm{Definition $\ref{2.2}$} \textit{by $(L, E, \varphi_{\mathbf{s}})$ to indicate the ordered basis $E$ and the generalized restricted map $\varphi_{\mathbf{s}}$.}
	
	$( \mathrm {ii} ) $ \textit{Any restricted Lie superalgebra  $(L,[p])$ is a GR Lie superalgebra, where  $\varphi_{\mathbf{s}}= \left.[p]\right|_{E}$ and $\mathbf{s}=(s_{i})_{i \in I}$ with $s_{i}=1, \forall i \in I$. }	
\end{rmk}

Next, we always assume $\mathfrak{g}=H(2,1,\underline{t})$ are Hamiltonian superalgebras, We have the following conclusions.
\begin{pro}
	Hamiltonian superalgebras are GR Lie superalgebras associated with an ordered basis
	\begin{align*}
	\{&D_H(x_1),D_H(x_2),\cdots,D_H(x_1^{(i_1)}x_2^{(i_2)}),\cdots , D_H(x_1^{(p^{t_1}-1)}),D_H(x_2^{(p^{t_2}-1)}),D_H(x_1^{(p^{t_1}-1)}x_2^{(p^{t_2}-1)})\\
	&\mid 0\le i_1\le p^{t_1}-1,\,0\le i_2\le p^{t_2}-1,\,i_1+i_2\ge 1\}
	\end{align*}
and generalized restricted map $\varphi _{\mathbf{s}}$ that satisfy
	$$\varphi_{\mathbf{s}}(D_H(x_1x_2))=D_H(x_1x_2),$$
	$$\varphi_{\mathbf{s}}(D_H(x_1^{(j_1)}x_2^{(j_2)}))=0\,\,\text{ for }j_1,j_2\ne 1\text{ at the same time},$$
	where $\mathbf{s}=(t_2,t_1,1,1,\cdots, 1)$.
\end{pro}
\begin{proof}
	From \cite{BS} it follows that $\mathfrak{g}_{\bar{0}}$ are GR Lie algebras. So we only need to show that
	\begin{align*}
	(\operatorname{ad}D_H(x_1^{(i_1)}x_2^{(i_2)}))^{p^{s_i}}(y)=(\operatorname{ad}D_H(x_1^{(i_1)}x_2^{(i_2)})^{\varphi_{\mathbf{s}}})(y)
	\end{align*}
	for all $y\in \mathfrak{g}_{\bar{1}},\,0\le i_1\le p^{t_1}-1,\,0\le i_2\le p^{t_2}-1.$
	Indeed, for one thing,
	\begin{align*}
	&[\underbrace{D_H(x_1x_2),\cdots,[D_H(x_1x_2)} _{p \,\text{times}},D_H(x_1^{(j_1)}x_2^{(j_2)}\xi)]\cdots]=(i_2-i_1)^pD_H(x_1^{(j_1)}x_2^{(j_2)}\xi),
	\end{align*}
	for another
	\begin{align*}
	[D_H(x_1x_2), D_H(x_1^{(i_1)}x_2^{(i_2)}\xi)]=(i_2-i_1)D_H(x_1^{(j_1)}x_2^{(j_2)}\xi).
	\end{align*}
	Thus $\varphi_{\mathbf{s}}(D_H(x_1x_2))=D_H(x_1x_2)$.
	And on the one hand, for $j_1=1, \,j_2=0, \,0\le i_1\le p^{t_1}-1,\,0\le i_2\le p^{t_2}-1$,
	\begin{align*}
[\underbrace{D_H(x_1),\cdots,[D_H(x_1)} _{p^{t_2} \,\text{times}},D_H(x_1^{(i_1)}x_2^{(i_2)}\xi)]\cdots]
	=&[\underbrace{D_H(x_1),\cdots,[D_H(x_1)} _{p^{t_2-1} \,\text{times}},D_H(x_1^{(i_1)}x_2^{(i_2-1)}\xi)]\cdots]\\
	=&\cdots=D_H(x_1^{(i_1)}x_2^{(i_2-p^{t_2})}\xi)=0,
	\end{align*}
	for $j_1=0, \,j_2=1, \,0\le i_1\le p^{t_1}-1,\,0\le i_2\le p^{t_2}-1$,
 \begin{align*}
 [\underbrace{D_H(x_2),\cdots,[D_H(x_2)} _{p^{t_1} \,\text{times}},D_H(x_1^{(i_1)}x_2^{(i_2)}\xi)]\cdots]
 =&[\underbrace{D_H(x_2),\cdots,[D_H(x_2)} _{p^{t_1-1} \,\text{times}},D_H(x_1^{(i_1-1)}x_2^{(i_2)}\xi)]\cdots]\\
 =&\cdots=D_H(x_1^{(i_1-p^{t_1})}x_2^{(i_2)}\xi)=0,
 \end{align*}
 	for $1\le j_1\le p^{t_1}-1, \,1\le j_2\le p^{t_1}-1 \text{ with } j_1,j_2\ne 1\text{ at the same time},\,0\le i_1\le p^{t_1}-1,\,0\le i_2\le p^{t_2}-1$,
  \begin{align*}
 &[ \underbrace{D_H(x_1^{(j_1)}x_2^{(j_2)}),\cdots ,[D_H(x_1^{(j_1)}x_2^{(j_2)})}_{p  \text{ times}},D_H(x_1^{(i_1)}x_2^{(i_2)}\xi)]\cdots]\\
 =&(C_{i_1+j_1-1}^{j_1-1}C_{i_2+j_2-1}^{j_2}-C_{i_1+j_1-1}^{j_1}C_{i_2+j_2-1}^{j_2-1})(C_{i_1+2j_1-2}^{j_1-1}C_{i_2+2j_2-2}^{j_2}-C_{i_1+2j_1-2}^{j_1}C_{i_2+2j_2-2}^{j_2-1})\\
 &\cdots(C_{i_1+kj_1-k}^{j_1-1}C_{i_2+kj_2-k}^{j_2}-C_{i_1+kj_1-k}^{j_1}C_{i_2+kj_2-k}^{j_2-1})\cdots(C_{i_1+pj_1-p}^{j_1-1}C_{i_2+pj_2-p}^{j_2}-C_{i_1+pj_1-p}^{j_1}C_{i_2+pj_2-p}^{j_2-1})\\
 &D_H(x_1^{(i_1+pj_1-p)}x_2^{(i_2+pj_2-p)}\xi)=0,
 \end{align*}
 on the other we have
 $[0,D_H(x_1^{(i_1)}x_2^{(i_2)}\xi)]=0.$
 Consequently, $\varphi_{\mathbf{s}}(D_H(x_1^{(j_1)}x_2^{(j_2)}))=0\,\,\text{ for }j_1,j_2\ne 1\text{ at the same time},$
 as desired.
\end{proof}

 \begin{rmk}
 	Hamiltonian superalgebras are restricted Lie superalgebras if and only if $t=1$.
 \end{rmk}
There is a natural $\mathbb{Z} $-grading structure on $\mathfrak{g} $, i.e., $\mathfrak{g}=\bigoplus_{i=-1}^{p^{t_1+t_2}-3} \mathfrak{g}_{[i]}$,
where
\begin{align*}
\mathfrak{g}_{[i]}={\rm span}_{\mathbb{F}}\{D_H(x_1^{(i_1)}x_2^{(i_2)}\xi^j)\mid i_1+i_2+j=i+2\},
\text{ for } 0\le i_1\le p^{t_1}-1, 0\le i_2\le p^{t_2}-1, j=0,1.
\end{align*}
Associated with this grading, one has the following filtration:
\begin{equation}\label{e2.2}
\mathfrak{g}=\mathfrak{g}_{-1} \supset \mathfrak{g}_{0} \supset \cdots \supset \mathfrak{g}_{p^{t_1+t_2}-3},
\end{equation}
where
\begin{align*}
\mathfrak{g}_{i}=\bigoplus_{j \geq i} \mathfrak{g}_{[j]} \text { for }-1 \leq i \leq p^{t_1+t_2}-3.
\end{align*}
In accordance with \cite{WQWLZ} the generators of the osp$(1|2)$ superalgebra read as follows:
\begin{align*}
	\{h:=E_{22}-E_{33},e:=E_{23},f:=E_{32}\mid E:=E_{13}+E_{21},F:=E_{12}-E_{31}\},
\end{align*}
and satisfy the following commutation and anticommutation relations:
$$[h,e]=e,\quad[h,f]=-f,\quad[h,E]=E,\quad[h,F]=-F,$$
$$[e,f]=h,\quad[e,E]=0,\quad[e,F]=-E,\quad[f,E]=-F,$$
$$[f,F]=0,\quad[E,E]=2e,\quad[E,F]=h,\quad[F,F]=-2f.$$

\begin{lem}
	The zero graded component $\mathfrak{g}_{[0]}\cong {\rm osp}(1|2)$.
\end{lem}
\begin{proof}
A straightforward computation yields the following map:
\begin{align*}
	\phi: \mathfrak{g}_{[0]}&\longrightarrow {\rm osp}(1|2)\\
	x_2D_2-x_1D_1&\longmapsto h\\
	x_1D_2&\longmapsto f\\
	x_2D_1&\longmapsto e\\
	\xi D_2-x_1D_3&\longmapsto F\\
	\xi D_1+x_2D_3&\longmapsto E.
\end{align*}
Then it is easy to see that $\mathfrak{g}_{[0]}\cong {\rm osp}(1|2)$, as desired.
	\end{proof}

Furthermore, we have the triangular decomposition $\mathfrak{g}_{[0]}=\mathfrak{n}^{-} \oplus \mathfrak{h} \oplus \mathfrak{n}^{+}$, where $\mathfrak{n}^{-}=\mathbb{F} x_1D_2+\mathbb{F}(\xi D_2-x_1D_3), \mathfrak{n}^{+}=\mathbb{F} x_2D_1+\mathbb{F}(\xi D_1+x_2D_3) \text { and } \mathfrak{h}=\mathbb{F}  (x_2D_2-x_1D_1)$. Set $ \mathfrak{b}=\mathfrak{h} \oplus \mathfrak{n}^{+}$ and $B=\mathfrak{b} \oplus \mathfrak{g}_{1}=\mathfrak{b}+\sum_{i \geq 1} \mathfrak{g}_{[i]}$. It is easy to check
that the subalgebras $\mathfrak{b}$ and $B$ are (generalized) restricted subalgebras of $\mathfrak{g}$.

\subsection{Generalized $\chi$-reduced modules of generalized restricted Lie superalgebras}
\
\newline
\indent
\begin{defi}{\rm(\cite{YFY0})}
	Let $L=L_{\overline{0}} \oplus L_{\overline{1}}$ be a GR Lie superalgebra, a $p$-character $\chi\in L^*$, then
	$$\mathrm{ht}(\chi)=\min\{i \geq-1 \mid \chi (L_{i})=0\}$$
	is called the height of $\chi$.
\end{defi}
In this paper, we study generalized $\chi$-reduced modules of $H(2,1,\underline{t})$ with ht$( \chi) =0$.
\begin{defi}{\rm(\cite{YFY0})}
	Let $L=L_{\overline{0}} \oplus L_{\overline{1}}$ be a GR Lie superalgebra, a $p$-character $\chi\in L^*$. Let $J_\chi$ be the ideal of the universal enveloping superalgebra $U(L)$ generated by $e_i^{p^{s_i}}-e_i^{\varphi_{\mathbf{s}}}-\chi(e_i)^{p^{s_i}},  i\in I$. If
	$$U_\chi((L,E,\varphi_{\mathbf{s}}))=U(L)/J_\chi,$$
	$U_\chi((L,E,\varphi_{\mathbf{s}}))$ is called the generalized $\chi$-reduced enveloping superalgebra of $L$.
\end{defi}
One usually denotes the generalized  $\chi$-reduced enveloping superalgebra  $U_{\chi}((L, E, \varphi_{\mathbf{s}}))$ by $U_{\chi}(L)$ for brevity. Then a generalized $\chi$-reduced $L$-module is a $U_\chi(L)$-module.	In the case $\chi =0$, a generalized $\chi$-reduced module is said to be a generalized restricted module, and we call $U_0(L)$ the generalized restricted enveloping superalgebra of $L$, denoted
by $u(L)$ for brevity.
\begin{defi}{\rm(\cite{YFY0})}\label{Def2.8}
	Let $L=L_{\overline{0}} \oplus L_{\overline{1}}$ be a GR Lie superalgebra, a $p$-character $\chi\in L^*$, $v\in M$, if
	$$e_{i}^{p^{s_{i}} } \cdot v-e_{i}^{\varphi_{\mathbf{s}}} \cdot v=   \chi(e_{i})^{p^{s_{i}}} v, \forall v \in M, i \in I,$$
	$L$-module $M$ is said to be a generalized  $\chi$-reduced one.
\end{defi}
\begin{rmk}
	Let $L$ is $\mathbb{Z}$-graded structure, then all generalized $\chi$-reduced $L$-modules constitute a full subcategory of the category of $L$-modules, denoted by $U_\chi(L)$-mod.
\end{rmk}

\section{Simple modules of Hamiltonian superalgebras}\label{3}
In this section, first, we study the restricted simple modules of $\mathfrak{g}_{[0]}$. Then we study the structure of simple modules of $\mathfrak{g}$ with $p$-characters of height $0$.
\subsection{Restricted simple modules of $\mathfrak{g}_{[0]}$}\label{sub1}
\
\newline
\indent
Recall that $\mathfrak{g}=\mathfrak{g}_{[-1]}\oplus\mathfrak{g}_0$ with $\mathfrak{g}_0=\mathfrak{g}_{[0]}\oplus\mathfrak{g}_1$ and $\mathfrak{g}_{[0]}\cong {\rm osp}(1|2)$. From \cite[Theorem 2.3(N. Jacobson), Section 2.2] {HR}, we know that ${\rm osp}(1|2)_{\bar{0}}={\rm sp}(2)$ is a restricted Lie algebra, obviously. Then $h\in \mathfrak{h}$ is semisimple(\cite[Definition, Section 2.3]{HR}). Further, we know
each finite dimensional restricted $\mathfrak{g}_{[0]}$-module $M$ is a weight module, i.e., $M=\oplus_{\lambda\in \mathfrak{h}^*}M_{\lambda}$, where $M_\lambda=\{m\in M\mid h\cdot m=\lambda(h)m,\forall h\in \mathfrak{h}\}$(\cite[Proposition 3.3(2), Section 2.3]{HR}). Furthermore, $M$ is a restricted module, hence $\lambda(h)^{p}=\lambda(h^{[p]})$ for any weight $\lambda$ of $M$ and $h\in \mathfrak{h}$. Accordingly, we have $\lambda(x_2D_2-x_1D_1)\in \mathbb{F}_{p}$. We regard such weights as restricted weights. Set $\lambda\in \mathbb{F}_p$ be the restricted weights. Then the isomorphism classes of restricted simple modules of $\mathfrak{g}_{[0]}$, denoted by $L^0(\lambda)$, are parameterized by $\lambda \in \mathbb{F}_{p}$. More specifically, set $$v_{\lambda,k,l}=(x_1D_2)^k(\xi D_2-x_1D_3)^l\otimes v_\lambda, 0\le k\le p-1,l=0,1.$$
For $\lambda=0$, simplify $v_{\lambda,0,0}$ as $v_0$, the simple restricted $\mathfrak{g}_{[0]}$-module $L^0(\lambda)$ is one-dimensional with a basis $v_0$ such that
\begin{align*}
(x_2D_2-x_1D_1)\cdot v_0=x_1D_2\cdot v_0=x_2D_1\cdot v_0=(\xi D_2-x_1D_3)\cdot v_0=(\xi D_1+x_2D_3)\cdot v_0=0.
\end{align*}
And for $\lambda \in \mathbb{F}_p$ with $\lambda \ne 0$, the
simple restricted $\mathfrak{g}_{[0]}$-module $L^0(\lambda)$ is $(2\lambda+1)$-dimensional with a basis $\{v_{\lambda,k,l}\mid 0\le k\le \lambda, l=0,1, k+l\le \lambda\}$
with the action of $\mathfrak{g}_{[0]}$ defined as follows:	
\begin{align*}
(x_2D_2-x_1D_1)\cdot v_{\lambda, k,0}=&(\lambda-2k) v_{\lambda,k,0},\\
(x_2D_2-x_1D_1)\cdot v_{\lambda, k,1}=&(\lambda-2k-1) v_{\lambda,k,1},\\
(\xi D_2-x_1D_3)\cdot v_{\lambda, k,0}=&v_{\lambda,k,1},\\
(\xi D_2-x_1D_3)\cdot v_{\lambda, k,1}=&-v_{\lambda,k+1,0},\\
x_1D_2\cdot v_{\lambda, k,0}=&v_{\lambda, k+1,0},\\
x_1D_2\cdot v_{\lambda, k,1}=&v_{\lambda, k+1,1},\\
x_2D_1\cdot v_{\lambda, k,0}=&k(\lambda+1-k)v_{\lambda, k-1,0},\\
x_2D_1\cdot v_{\lambda, k,1}=&k(\lambda-k)v_{\lambda, k-1,1},\\
(\xi D_1+x_2D_3)\cdot v_{\lambda, k,0}=&kv_{\lambda,k-1,1},\\
(\xi D_1+x_2D_3)\cdot v_{\lambda, k,1}=&(\lambda-k) v_{\lambda,k,0}.
\end{align*}
We specify that $v_{\lambda,k,0}=0$ when $k<0$ or $k>\lambda$ and $v_{\lambda,k,1}=0$ when $k<0$ or $k\ge\lambda$.
\subsection{Simple modules of Hamiltonian superalgebras with $p$-characters of height 0}
\

 Set $\chi\in \mathfrak{g}^{*}$ with ${\rm ht}(\chi)=0$, i.e., $\chi(\mathfrak{g}_0)=0$ and $\chi(\mathfrak{g}_{[-1]})\ne 0$. We make the convention that any linear function $\chi\in \mathfrak{g}^*$ in this paper satisfies $\chi(\mathfrak{g}_{\bar{1}})=0$.
Under the action of the automorphism group of $\mathfrak{g}$, we can take $\chi$ as one of
 the following two types without loss of generality:

 \textbf{Type I\;:} $\chi(D_1)=\chi(D_2)=1$ and $\chi(D_3)=\chi(\mathfrak{g}_0)=0$.
 
  \textbf{Type II\;:} $\chi(D_1)=1$ and $\chi(D_2)=\chi(D_3)=\chi(\mathfrak{g}_0)=0$.

 \textbf{Type III\;:} $\chi(D_2)=1$ and $\chi(D_1)=\chi(D_3)=\chi(\mathfrak{g}_0)=0$.

Then each simple $U_\chi(\mathfrak{g}_0)$-module is a restricted $\mathfrak{g}_{[0]}$-module $L^0(\lambda)$ with trivial $\mathfrak{g}_1$-action for some $\lambda \in \mathbb{F}_p$. We define generalized $\chi$-reduced Kac $\mathfrak{g}$-modules $K_\chi(\lambda)$ as the induced $U_\chi(\mathfrak{g})$-module, which are $U_\chi(\mathfrak{g})\otimes_{u(\mathfrak{g}_0)}L^0(\lambda)$. There is the following intuitive observation.
\begin{lem}
	Each simple $U_\chi(\mathfrak{g})$-module is a quotient of a generalized $\chi$-reduced Kac $\mathfrak{g}$-module $K_\chi(\lambda)$ for some $\lambda \in \mathbb{F}_p$.
\end{lem}
\begin{proof}
	Set $N$ be a simple $U_\chi(\mathfrak{g})$-module. Consider $N$ as a module over $U_\chi(\mathfrak{g}_0)=u(\mathfrak{g}_0)$, and
	take a simple $u(\mathfrak{g}_0)$-submodule $N_0$. Then we have an isomorphism $\varphi: L^0(\lambda)\longrightarrow N_0$ for some $\lambda \in \mathbb{F}_p$. The universal property of tensor products shows that $\varphi$ induces the following homomorphism of $\mathfrak{g}$-modules:
	\begin{align*}
	\varPhi: K_\chi(\lambda)=U_\chi(\mathfrak{g})\otimes_{u(\mathfrak{g}_0)}L^0(\lambda)&\longrightarrow N\\
	D_1^{i_1}D_2^{i_2}D_3^{j}\otimes v &\longmapsto D_1^{i_1}D_2^{i_2}D_3^{j}\cdot \varphi(v)
	\end{align*}
	for $j=0,1$. Since $N$ is simple, $\varPhi$ is surjective, i.e., $N$ is a quotient of $K_\chi(\lambda)$, as desired.
\end{proof}
For $\lambda \in \mathbb{F}_p$, a nonzero vector $v$ in a $\mathfrak{g}$-module $M$ is defined as a singular vector of weight
$\lambda$ if $h\cdot v=\lambda(h)v$ and $u\cdot v=0$ for all $h\in \mathfrak{h}$ and $u\in \mathfrak{n}^++\mathfrak{g}_1$.
From the preceding analysis, the following conclusions can be drawn.

 \textbf{Type I\;:} $\chi(D_1)=\chi(D_2)=1$ and $\chi(D_3)=\chi(\mathfrak{g}_0)=0$.\label{T1}
 \begin{lem}\label{lem3.2}
 	For $\lambda \in \mathbb{F}_p$, $\chi(D_1)=\chi(D_2)=1$ and $\chi(D_3)=\chi(\mathfrak{g}_0)=0$, generalized $\chi$-reduced
 	Kac $\mathfrak{g}$-modules $K_\chi(\lambda)$ are simple.
 \end{lem}
\begin{proof}
In this type, we divide the different values of $\lambda$ into two cases.

\textbf{Case 1\;: $\lambda = 0$}.

In this case, $L^0(0)=\mathbb{F}v_0$ is one-dimensional, and generalized $\chi$-reduced Kac $\mathfrak{g}$-modules
\begin{align*}
K_\chi(0)={\rm span}_\mathbb{F}\{D_1^{i_1}D_2^{i_2}D_3^j\otimes v_0\mid 0\le i_1\le p^{t_1}-1,\, 0\le i_2\le p^{t_2}-1,\,j=0,1\}.
\end{align*}
	Let $M$ be a nonzero submodule of $K_\chi(0)$. Let $v\in M$ be a nonzero
	weight vector with respect to the Cartan subalgebra $\mathfrak{h}$. Note that
	\begin{align*}
	(x_2D_2-x_1D_1)\cdot D_1^{i_1}D_2^{i_2}\otimes v_0=(i_1-i_2)D_1^{i_1}D_2^{i_2}\otimes v_0 \text{ for } 0\le i_1\le p^{t_1}-1,\,0\le i_2\le p^{t_2}-1
	\end{align*}
	and
	\begin{align*}
	(x_2D_2-x_1D_1)\cdot D_1^{i_1}D_2^{i_2}D_3\otimes v_0=(i_1-i_2)D_1^{i_1}D_2^{i_2}D_3\otimes v_0 \text{ for } 0\le i_1\le p^{t_1}-1,\,0\le i_2\le p^{t_2}-1.
	\end{align*}
	Set $N={\rm min}\{p^{t_1}-i_1-1, p^{t_2}-i_2-1\}$, let $0\ne v\in M$ be a weight vector. Then by a direct analysis of the weight of $D_1^{i_1}D_2^{i_2}D_3^j\otimes v_0$ and by applying appropriate action of $D_1^{s_1}$, $D_2^{s_2}$ and $D_3^{t}$
	to $v$, we can assume that
	\begin{align*}
	v=\sum_{j=0}^{N}a_{i_1,i_2,j}D_1^{i_1+j}D_2^{i_2+j}\otimes v_0+\sum_{j=0}^{N}b_{i_1,i_2,j}D_1^{i_1+j}D_2^{i_2+j}D_3\otimes v_0
	\end{align*}
	with $a_{i_1,i_2,N},\,b_{i_1,i_2,N}\ne0$.
	Then we have
	\begin{align*}
	&(x_1^{(i_1+N-1)}x_2^{(i_2+N)}\xi D_2-x_1^{(i_1+N)}x_2^{(i_2+N-1)}\xi D_1-x_1^{(i_1+N)}x_2^{(i_2+N)}D_3)\cdot v\\
	=&\sum_{j=0}^{N}(-1)^{l_1+l_2}a_{i_1,i_2,j}C_{i_1+j}^{l_1}C_{i_2+j}^{l_2}D_1^{i_1+j-l_1}D_2^{i_2+j-l_1}(x_1^{(i_1+N-l_1-1)}x_2^{(i_2+N-l_2)}\xi D_2-x_1^{(i_1+N-l_1)}x_2^{(i_2+N-l_2-1)}\xi D_1\\
	&-x_1^{(i_1+N-l_1)}x_2^{(i_2+N-l_2)}D_3)\otimes v_{0}\\
	&+\sum_{j=0}^{N}(-1)^{l_1+l_2}b_{i_1,i_2,j}C_{i_1+j}^{l_1}C_{i_2+j}^{l_2}D_1^{i_1+j-l_1}D_2^{i_2+j-l_1}(x_1^{(i_1+N-l_1-1)}x_2^{(i_2+N-l_2)}\xi D_2-x_1^{(i_1+N-l_1)}x_2^{(i_2+N-l_2-1)}\xi D_1\\
	&-x_1^{(i_1+N-l_1)}x_2^{(i_2+N-l_2)}D_3)D_3\otimes v_{0}\\
	=&(-1)^{i_1+i_2}a_{i_1,i_2,N}C_{i_1+N}^{i_1+N}C_{i_2+N}^{i_2+N}D_3\otimes v_{0}+(-1)^{i_1+i_2-1}b_{i_1,i_2,N}C_{i_1+N}^{i_1+N-1}D_1(\xi D_2-x_1D_3)D_3\otimes v_{0}\\
	&-(-1)^{i_1+i_2-1}b_{i_1,i_2,N}C_{i_2+N}^{i_2+N-1}D_2(\xi D_1+x_2D_3)D_3\otimes v_{0}\\
	=&(-1)^{i_1+i_2}a_{i_1,i_2,N}D_3\otimes v_{0}+(-1)^{i_1+i_2}b_{i_1,i_2,N}(i_2-i_1)D_1D_2\otimes v_{0}\in M.
	\end{align*}
	When $i_1\ne i_2$,
	\begin{align*}
	&x_2D_1\cdot ((-1)^{i_1+i_2}a_{i_1,i_2,N}D_3\otimes v_{0}+(-1)^{i_1+i_2}b_{i_1,i_2,N}(i_2-i_1)D_1D_2\otimes v_{0})\\
	=&-(-1)^{i_1+i_2}b_{i_1,i_2,N}(i_2-i_1)D_1^2\otimes v_{0}\in M.
	\end{align*}
	It is clear that
	\begin{align*}
	D_1^{p^{t_1}-2}\cdot D_1^2\otimes v_{0}=\chi(D_1)^{p^{t_1}}\otimes v_{0}=1\otimes v_{0}\in M.
	\end{align*}
	When $i_1= i_2$, $D_3\otimes v_{0}\in M$.
	By
	\begin{align*}
	(\xi D_1+x_2D_3)\cdot D_3\otimes v_{0}=D_1\otimes v_{0}\in M,
	\end{align*}
	we have
	\begin{align*}
	D_1^{p^{t_1}-1}\cdot D_1\otimes v_{0}=\chi(D_1)^{p^{t_1}}\otimes v_{0}=1\otimes v_{0}\in M.
	\end{align*}
	Therefore, $M=K_\chi(0)$. Consequently, $K_\chi(0)$ are simple.
	
\textbf{Case 2\;: $\lambda \ne0$}.

In this case,  the
simple restricted $\mathfrak{g}_{[0]}$-module $L^0(\lambda)$ is $(2\lambda+1)$-dimensional with a basis $\{v_{\lambda,k,l}\mid 0\le k\le\lambda, l=0,1, k+l\le \lambda\}$
with the action of $\mathfrak{g}_{[0]}$ defined as in section $\ref{sub1}$. Then generalized $\chi$-reduced Kac $\mathfrak{g}$-modules
$$K_\chi(\lambda)={\rm span}_\mathbb{F}\{D_1^{i_1}D_2^{i_2}D_3^j\otimes v_{\lambda,k,l}\mid 0\le i_1\le p^{t_1}-1, \,0\le i_2\le p^{t_2}-1,\,j,l=0,1,\,0\le k\le \lambda,\,k+l\le\lambda \}.$$

Let $M$ be a nonzero submodule of $K_\chi(\lambda)$. Set $0\ne w\in M$ to be a weight vector. Maintain $N={\rm min}\{p^{t_1}-i_1-1, p^{t_2}-i_2-1\}$. Then by a direct analysis of the weight of $D_1^{i_1}D_2^{i_2}D_3^j\otimes v_{\lambda,k,l}$ and by applying appropriate action of $D_1^{s_1}$, $D_2^{s_2}$ and $D_3^{t}$
to $w$, we can assume that
\begin{align*}
w=&\sum_{j=0}^{N}c_{i_1,i_2,j,k}D_1^{i_1+j}D_2^{i_2+j}\otimes v_{\lambda,k,0}+\sum_{j=0}^{N}d_{i_1,i_2,j,k}D_1^{i_1+j}D_2^{i_2+j}D_3\otimes v_{\lambda,k,0}
\end{align*}
or
\begin{align*}
w=&\sum_{j=0}^{N}e_{i_1,i_2,j,k}D_1^{i_1+j}D_2^{i_2+j}\otimes v_{\lambda,k,1}+\sum_{j=0}^{N}f_{i_1,i_2,j,k}D_1^{i_1+j}D_2^{i_2+j}D_3\otimes v_{\lambda,k,1}
\end{align*}
for $0\le i_1\le p^{t_1}-1, \,0\le i_2\le p^{t_2}-1$, $0\le k\le \lambda$, $c_{i_1,i_2,N,k},\,c_{i_1,i_2,N-1,k},\,e_{i_1,i_2,N,k},\,e_{i_1,i_2,N-1,k}\ne 0$.  We will now consider two different subcases.

\textbf{Subcase 2.1\;:}
 \begin{align*}
w=&\sum_{j=0}^{N}c_{i_1,i_2,j,k}D_1^{i_1+j}D_2^{i_2+j}\otimes v_{\lambda,k,0}+\sum_{j=0}^{N}d_{i_1,i_2,j,k}D_1^{i_1+j}D_2^{i_2+j}D_3\otimes v_{\lambda,k,0}.
\end{align*}
Clearly,
\begin{align*}
D_3\cdot w=\sum_{j=0}^{N}c_{i_1,i_2,j,k}D_1^{i_1+j}D_2^{i_2+j}D_3\otimes v_{\lambda,k,0}\in M.
\end{align*}
Then
\begin{align*}
	&(x_1^{(i_1+N-1)}x_2^{(i_2+N)}\xi D_2-x_1^{(i_1+N)}x_2^{(i_2+N-1)}\xi D_1-x_1^{(i_1+N)}x_2^{(i_2+N)}D_3)\cdot \sum_{j=0}^{N}c_{i_1,i_2,j,k}D_1^{i_1+j}D_2^{i_2+j}D_3\otimes v_{\lambda,k,0}\\
	=&\sum_{j=0}^{N}(-1)^{l_{1}+l_{2}} e_{i_1, i_2, j, k} C_{i_1+j}^{l_1} C_{i_2+j}^{l_2} D_1^{i_1+j-l_1} D_2^{i_2+j-l_2}(x_1^{(i_1+N-l_1-1)} x_2^{(i_2+N-l_2)} \xi D_2\\
	&-x_1^{(i_1+N-l_{1})} x_{2}^{(i_{2}+N-l_{2}-1)} \xi D_{1}-x_1^{(i_1+N-l_1)} x_2^{(i_2+N-l_2)} D_3) D_3 \otimes v_{\lambda, k, 0}\\
	=&(-1)^{i_1+i_2}(\lambda-2k)c_{i_1,i_2,N-1,k}(1\otimes v_{\lambda,k,0}+D_3\otimes v_{\lambda,k,0})\\
&+(-1)^{i_1+i_2}(\lambda-2k)(c_{i_1,i_2,N,k}(i_1+N)(i_2+N)+i_2-i_1)D_1D_2\otimes v_{\lambda,k,0}\\
	&+(-1)^{i_1+i_2}(\lambda-2k)c_{i_1,i_2,N,k}(i_1+N)(i_2+N)D_1D_2D_3\otimes v_{\lambda,k,0}\\
		&+(-1)^{i_1+i_2}c_{i_1,i_2,N,k}((i_1+N)D_1D_3\otimes v_{\lambda,k,1}-k(i_2+N)D_2D_3\otimes v_{\lambda,k-1,1})\\
	&+(-1)^{i_1+i_2}c_{i_1,i_2,N,k}(C_{i_1+N}^2D_1^2\otimes v_{\lambda,k+1,0}-k(\lambda+1-k)C_{i_2+N}^2D_2^2\otimes v_{\lambda,k-1,0})\in M.
\end{align*}
For simplicity, let $\frac{c_{i_1,i_2,N-1,k}}{c_{i_1,i_2,N,k}}$ be denoted as $a_N$, and the resulting equation obtained by removing the non-zero identical coefficients from the above equation is denoted as $w_1$. Then we have
\begin{align*}
	D_3\cdot w_1=&((\lambda-2k)(i_1+N)(i_2+N)+i_2-i_1)D_1D_2D_3\otimes v_{\lambda,k,0}+a_N(\lambda-2k)D_3\otimes v_{\lambda,k,0}\\
	&+C_{i_1+N}^2D_1^2D_3\otimes v_{\lambda,k+1,0}-k(\lambda+1-k)C_{i_2+N}^2D_2^2D_3\otimes v_{\lambda,k-1,0}\in M.
\end{align*}
We denote the above equation as $w_1'$.

\textbf{Subcase 2.1.1\;:} $1\le k\le \lambda-1$ and $k\ne\frac{\lambda}{2}$.

We have
\begin{align*}
	x_1^{(2)}D_2\cdot w_1'=&-((\lambda-2k)(i_1+N)(i_2+N)+i_2-i_1-C_{i_1+N}^2)D_2D_3\otimes v_{\lambda,k+1,0}\\
	&-C_{i_1+N}^2D_1D_3\otimes v_{\lambda,k+2,0}\in M.
\end{align*}
and
\begin{align*}
&x_2^{(2)}D_1\cdot(((\lambda-2k)(i_1+N)(i_2+N)+i_2-i_1-C_{i_1+N}^2)D_2D_3\otimes v_{\lambda,k+1,0}+C_{i_1+N}^2D_1D_3\otimes v_{\lambda,k+2,0})\\
	=&-((\lambda-2k)(i_1+N)(i_2+N)+i_2-i_1-C_{i_1+N}^2)(\lambda-k)(k+1)D_3\otimes v_{\lambda,k,0}\in M.
\end{align*}
Therefore, we can get $D_3\otimes v_{\lambda,k,0}\in M$. Then from
\begin{align*}
	x_2D_1\cdot D_3\otimes v_{\lambda,k,0}=D_3\otimes x_2D_1\cdot v_{\lambda,k,0}=k(\lambda+1-k)D_3\otimes v_{\lambda,k-1,0}\in M
\end{align*}
that $D_3\otimes v_{\lambda,k-1,0}\in M$. By inductively applying $x_2D_1$, we obtain $D_3\otimes v_{\lambda,0,0}\in M$. Then 
\begin{align*}
	(\xi D_1+x_2D_3)\cdot D_3\otimes v_{\lambda,0,0}=D_1\otimes v_{\lambda,0,0}\in M.
\end{align*}
It is obvious that
\begin{align*}
D_1^{p^{t_1}-1}	\cdot D_1\otimes v_{\lambda,0,0}=\chi(D_1)^{p^{t_1}}\otimes v_{\lambda,0,0}=1\otimes v_{\lambda,0,0}\in M.
\end{align*}
\textbf{Subcase 2.1.2\;:} $k=0$.

By
\begin{align*}
	x_1^{(2)}D_2\cdot w_1'=-(\lambda(i_1+N)(i_2+N)+i_2-i_1-C_{i_1+N}^2)D_2D_3\otimes v_{\lambda,1,0}\in M,
\end{align*}
we can get $D_2D_3\otimes v_{\lambda,1,0}\in M$. Clearly,
\begin{align*}
	x_2^{(2)}D_1\cdot D_2D_3\otimes v_{\lambda,1,0}=-\lambda D_3 \otimes v_{\lambda,0,0}\in M.
\end{align*}
It follows from
\begin{align*}
	(\xi D_1+x_2D_3)\cdot D_3\otimes v_{\lambda,0,0}=D_1\otimes v_{\lambda,0,0}\in M
\end{align*}
that we can get
\begin{align*}
	D_1^{p^{t_1}-1}\cdot D_1\otimes v_{\lambda,0,0}=\chi(D_1)^{p^{t_1}}\otimes v_{\lambda,0,0}=1\otimes v_{\lambda,0,0}\in M.
\end{align*}

\textbf{Subcase 2.1.3\;:} $k=\frac{\lambda}{2}$.

In this subcase, $\lambda$ is even.
For $i_1\ne i_2$, we have
\begin{align*}
	x_1^{(2)}D_2\cdot w_1'=(i_1-i_2+C_{i_1+N}^2)D_2D_3\otimes v_{\lambda,\frac{\lambda+2}{2},0}-C_{i_1+N}^2D_1D_3\otimes v_{\lambda,\frac{\lambda+4}{2},0}\in M.
\end{align*}
Application of $x_2^{(2)}D_1$ to the right-hand side expression of the above equation leads to $D_3\otimes v_{\lambda,\frac{\lambda}{2},0}\in M$. 
By discussing it similarly to the previous, we immediately get $1\otimes v_{\lambda,0,0}\in M$.

For $i_1=i_2$, we have
\begin{align*}
	(x_1^{(2)}\xi D_2-x_1^{(2)}D_3)\cdot w_1'=C_{i_1+N}^2(D_2\otimes v_{\lambda,\frac{\lambda+2}{2},0}-D_3\otimes v_{\lambda,\frac{\lambda+2}{2},1}-D_1\otimes v_{\lambda,\frac{\lambda+4}{2},0})\in M.
\end{align*}
Then
\begin{align*}
	x_2^{(2)}D_1\cdot(D_2\otimes v_{\lambda,\frac{\lambda+2}{2},0}-D_3\otimes v_{\lambda,\frac{\lambda+2}{2},1}-D_1\otimes v_{\lambda,\frac{\lambda+4}{2},0})=\frac{1}{4}\lambda(\lambda+2)\otimes v_{\lambda,\frac{\lambda}{2},0}\in M.
\end{align*}
Hence, we can get $1\otimes v_{\lambda,\frac{\lambda}{2},0}\in M$. It follows from
\begin{align*}
	x_2D_1\cdot1 \otimes v_{\lambda,\frac{\lambda}{2},0}=\frac{1}{4}\lambda(\lambda+2)\otimes v_{\lambda,\frac{\lambda-2}{2},0}\in M,
\end{align*}
that $1\otimes v_{\lambda,\frac{\lambda-2}{2},0}\in M$. 
By continuously acting on the right-hand side of the above equation with $x_2D_1$, we readily obtain $1\otimes v_{\lambda,0,0}\in M$. 

\textbf{Subcase 2.1.4\;:} $k=\lambda$.

It is plain to see that
\begin{align*}
	x_2^{(2)}D_1\cdot w_1'=\lambda C_{i_1+N}^2(D_1D_3\otimes v_{\lambda,\lambda-1,0}-\lambda D_2D_3\otimes v_{\lambda,\lambda-2,0})\in M.
	\end{align*}
Then we have
\begin{align*}
	&(x_1^{(2)}\xi D_2-x_1^{(3)}D_3)\cdot(D_1D_3\otimes v_{\lambda,\lambda-1,0}-\lambda D_2D_3\otimes v_{\lambda,\lambda-2,0})\\
	=&-(x_1\xi D_2-x_1^{(2)}D_3)D_3\otimes v_{\lambda,\lambda-1,0}\\
	=&-1\otimes x_1D_2\cdot v_{\lambda,\lambda-1,0}\\
	=&-1\otimes v_{\lambda,\lambda,0}\in M.
\end{align*}
Applying $x_2D_1$ successively to $1\otimes v_{\lambda,\lambda,0}$ leads directly to $1\otimes v_{\lambda,0,0}\in M$.

\textbf{Subcase 2.2\;:} 
\begin{align*}
	w=&\sum_{j=0}^{N}e_{i_1,i_2,j,k}D_1^{i_1+j}D_2^{i_2+j}\otimes v_{\lambda,k,1}+\sum_{j=0}^{N}f_{i_1,i_2,j,k}D_1^{i_1+j}D_2^{i_2+j}D_3\otimes v_{\lambda,k,1}.
\end{align*}
It is evident that
\begin{align*}
	D_3\cdot w=\sum_{j=0}^{N}e_{i_1,i_2,j,k}D_1^{i_1+j}D_2^{i_2+j}D_3\otimes v_{\lambda,k,0}\in M.
\end{align*}
Then
\begin{align*}
	&(x_1^{(i_1+N-1)}x_2^{(i_2+N)}\xi D_2-x_1^{(i_1+N)}x_2^{(i_2+N-1)}\xi D_1-x_1^{(i_1+N)}x_2^{(i_2+N)}D_3)\cdot \sum_{j=0}^{N}c_{i_1,i_2,j,k}D_1^{i_1+j}D_2^{i_2+j}D_3\otimes v_{\lambda,k,1}\\
	=&\sum_{j=0}^{N}(-1)^{l_{1}+l_{2}} e_{i_1, i_2, j, k} C_{i_1+j}^{l_1} C_{i_2+j}^{l_2} D_1^{i_1+j-l_1} D_2^{i_2+j-l_2}(x_1^{(i_1+N-l_1-1)} x_2^{(i_2+N-l_2)} \xi D_2\\
		&-x_1^{(i_1+N-l_{1})} x_{2}^{(i_{2}+N-l_{2}-1)} \xi D_{1}-x_1^{(i_1+N-l_1)} x_2^{(i_2+N-l_2)} D_3) D_3 \otimes v_{\lambda, k, 1}\\
		=&(-1)^{i_1+i_2}(\lambda-2k-1)(-1)^{i_1+i_2}e_{i_1,i_2,N-1,k}(1\otimes v_{\lambda,k,1}+D_3\otimes v_{\lambda,k,1})\\
		&+(-1)^{i_1+i_2}(\lambda-2k-1)(e_{i_1,i_2,N,k}(i_1+N)(i_2+N)+i_2-i_1)	D_1D_2\otimes v_{\lambda,k,1}\\
		&+(-1)^{i_1+i_2}(\lambda-2k-1)e_{i_1,i_2,N,k}(i_1+N)(i_2+N)D_1D_2D_3\otimes v_{\lambda,k,1}\\
		&+(-1)^{i_1+i_2}e_{i_1,i_2,N,k}((i_1+N)D_1D_3\otimes v_{\lambda,k+1,0}-(\lambda-k)(i_2+N)D_2D_3\otimes v_{\lambda,k,0})\\
		&+(-1)^{i_1+i_2}e_{i_1,i_2,N,k}(C_{i_1+N}^2D_1^2\otimes v_{\lambda,k+1,1}-k(\lambda-k)C_{i_2+N}^2D_2^2\otimes v_{\lambda,k-1,1})\in M.
\end{align*}
To simplify matters, we denote $\frac{e_{i_1,i_2,N-1,k}}{e_{i_1,i_2,N,k}}$
 by $b_N$, and the resulting equation, obtained by removing the non-zero identical coefficients from the preceding equation, is designated as $w_2$. Then we have
\begin{align*}
	D_3\cdot w_2=&((\lambda-2k-1)(i_1+N)(i_2+N)+i_2-i_1)D_1D_2D_3\otimes v_{\lambda,k,1}+b_N(\lambda-2k-1)D_3\otimes v_{\lambda,k,1}\\
	&+C_{i_1+N}^2D_1^2D_3\otimes v_{\lambda,k+1,1}-k(\lambda-k)C_{i_2+N}^2D_2^2D_3\otimes v_{\lambda,k-1,1}\in M.
\end{align*}
We denote the above equation as $w_2'$. We will now consider three different subcases for the different values of $k$.

\textbf{Subcase 2.2.1\;:} $1\le k\le \lambda-1$ and $k\ne\frac{\lambda-1}{2}$.

It is obvious that
\begin{align*}
	x_1^{(2)}D_2\cdot w_2'=&-((\lambda-2k-1)(i_1+N)(i_2+N)+i_2-i_1-C_{i_1+N}^2)D_2D_3\otimes v_{\lambda,k+1,1}\\
	&-C_{i_1+N}^2D_1D_3\otimes v_{\lambda,k+2,1}\in M
\end{align*}
and
\begin{align*}
	&x_2^{(2)}D_1\cdot(((\lambda-2k-1)(i_1+N)(i_2+N)+i_2-i_1-C_{i_1+N}^2)D_2D_3\otimes v_{\lambda,k+1,1}+C_{i_1+N}^2D_1D_3\otimes v_{\lambda,k+2,1})\\
	=&-((\lambda-2k-1)(i_1+N)(i_2+N)+i_2-i_1-C_{i_1+N}^2)(\lambda-k-1)(k+1)D_3\otimes v_{\lambda,k,1}\in M.
\end{align*}
Therefore, we can get $D_3\otimes v_{\lambda,k,1}\in M$. According to
\begin{align*}
	x_2D_1\cdot D_3\otimes v_{\lambda,k,1}=D_3\otimes x_2D_1\cdot v_{\lambda,k,1}=k(\lambda-k)D_3\otimes v_{\lambda,k-1,1}\in M,
\end{align*}
it follows that $D_3\otimes v_{\lambda,k-1,1}\in M$. By inductively applying $x_2D_1$, we obtain $D_3\otimes v_{\lambda,0,1}\in M$. Then we have
\begin{align*}
	(\xi D_2-x_1D_3)\cdot D_3\otimes v_{\lambda,0,1}=D_2\otimes v_{\lambda,0,1}+D_3\otimes v_{\lambda,1,0}\in M
\end{align*}
and
\begin{align*}
	(x_2^{(2)}D_2-x_1x_2D_1)\cdot (D_2\otimes v_{\lambda,0,1}+D_3\otimes v_{\lambda,1,0})=(1-\lambda)\otimes v_{\lambda,0,1}\in M.
\end{align*}
Thus, $1\otimes v_{\lambda,0,1}\in M$. Furthermore,
\begin{align*}
	(\xi D_1+x_2D_3)\cdot 1\otimes v_{\lambda,0,1}=\lambda\otimes v_{\lambda,0,0}\in M.
\end{align*} 
This implies that $1\otimes v_{\lambda,0,0}\in M$.

\textbf{Subcase 2.2.2\;:} $k=0$.

For  $\lambda\ne 1$, following a similar approach as in Subcase 2.2.1, we obtain $1\otimes v_{\lambda,0,0}\in M$.

For $\lambda=1$, when $i_1\ne i_2$, $w_2'=(i_2-i_1)D_1D_2D_3\otimes v_{\lambda,0,1}\in M$. It is clear that 
$$D_1^{p^{t_1}-1}D_2^{p^{t_2}-1}\cdot D_1D_2D_3\otimes v_{1,0,1}=\chi(D_1)^{p^{t_1}}\chi(D_2)^{p^{t_2}}D_3\otimes v_{1,0,1}=D_3\otimes v_{1,0,1}\in M.$$ By discussing it similarly to the previous, we can get $1\otimes v_{1,0,0}\in M$. When $i_1= i_2$, $w_2=(i_1+N)D_1D_3\otimes v_{1,1,0}\in M$.
Further, we can also get $1\otimes v_{1,0,0}\in M$.

\textbf{Subcase 2.2.3\;:} $k=\frac{\lambda-1}{2}$.

When $i_1\ne i_2$, by discussng similarly as before we can also obtain $1\otimes v_{\lambda,0,0}\in M$.
When $i_1= i_2$, by $x_1^{(3)}D_2\cdot w_2'=C_{i_1+N}^2D_3\otimes v_{\lambda,\frac{\lambda+3}{2},1}\in M$, we can also obtain $1\otimes v_{\lambda,0,0}\in M$.

This implies that $1\otimes v_{\lambda,0,0}\in M$ for $\lambda\in \mathbb{F}_p$ and $0\le k\le \lambda$. Therefore, we get $M=K_\chi(\lambda)$. In conclusion, $K_\chi(\lambda)$ are simple, as desired.
\end{proof}

  \textbf{Type II\;:} $\chi(D_1)=1$ and $\chi(D_2)=\chi(D_3)=\chi(\mathfrak{g}_0)=0$.
   \begin{lem}\label{lem3.3}
  	For $\lambda \in \mathbb{F}_p$, $\chi(D_1)=1$ and $\chi(D_2)=\chi(D_3)=\chi(\mathfrak{g}_0)=0$, generalized $\chi$-reduced
  	Kac $\mathfrak{g}$-modules $K_\chi(\lambda)$ are simple.
  \end{lem}
  \begin{proof}
  	 In this type, what is different from Type I is $0=D_2^{p^{t_2}-1}\cdot D_2\otimes v_{\lambda,k,l}\ne1\otimes v_{\lambda,k,l}$. Since $\chi(D_2)$ is exclusively employed in Subcase 2.2.2, it suffices to examine whether $D_2D_3\otimes v_{1,0,1}\in M$ yields $1\otimes v_{1,0,0}\in M$. Thus, we have
  	 $x_2D_1\cdot D_2D_3\otimes v_{1,0,1}=-D_1D_3\otimes v_{1,0,1}\in M$. By setting $\chi(D_1)=1$, we first obtain $D_3\otimes v_{1,0,1}\in M$, which subsequently yields $1\otimes v_{1,0,0}\in M$ through algebraic manipulation. In conclusion, $K_\chi(\lambda)$ are simple, as desired.
  \end{proof}

\textbf{Type III\;:} $\chi(D_2)=1$ and $\chi(D_1)=\chi(D_3)=\chi(\mathfrak{g}_0)=0$.
  \begin{lem}\label{lem3.4}
	For $\lambda \in \mathbb{F}_p$, $\chi(D_2)=1$ and $\chi(D_1)=\chi(D_3)=\chi(\mathfrak{g}_0)=0$, generalized $\chi$-reduced
	Kac $\mathfrak{g}$-modules $K_\chi(\lambda)$ are simple.
\end{lem}
\begin{proof}
In this type, what is different from Type I is $0=D_1^{p^{t_1}-1}\cdot D_1\otimes v_{\lambda,k,l}\ne1\otimes v_{\lambda,k,l}$. For $\lambda=0$, we have
	$x_1D_2\cdot D_1\otimes v_0=-D_2\otimes v_0\in M.$ As a result, $$D_2^{p^{t_2}-1}\cdot D_2\otimes v_0=\chi(D_2)^{p^{t_2}}\otimes v_0=1\otimes v_0\in M.$$
	
	For $1\le\lambda\le p-1$, it suffices to consider the case
	$D_1\otimes v_{\lambda,0,0}$ and $D_1D_3\otimes v_{1,0,1}$. Since
	$$x_1^{(2)}D_2\cdot D_1\otimes v_{\lambda,0,0}=-1\otimes x_1D_2\cdot v_{\lambda,0,0}=-1\otimes v_{\lambda,1,0}\in M$$
	and
	$$x_2D_1\cdot1\otimes v_{\lambda,1,0}=\lambda\otimes v_{\lambda,0,0}\in M,$$
	we can get $1\otimes v_{\lambda,0,0}\in M$. From
$$x_1D_2\cdot D_1D_3\otimes v_{1,0,1}=-D_2D_3\otimes v_{1,0,1}\in M$$
 and $\chi(D_2)=1$ that $D_3\otimes v_{1,0,1}\in M$. Furthermore, we can get $1\otimes v_{1,0,0}\in M$.
	 In a word, $K_\chi(\lambda)$ are simple, as desired.
\end{proof}
From Lemmas \ref{lem3.2}-\ref{lem3.4}, the following conclusions can be clearly obtained.
 \begin{pro}\label{pro3.5}
 	Let $\chi\in\mathfrak{g}^*$ with ${\rm ht}(\chi)=0$.
 	For $\lambda \in \mathbb{F}_p$, generalized $\chi$-reduced
 	Kac $\mathfrak{g}$-modules $K_\chi(\lambda)$ are simple.
 \end{pro}

\section{Isomorphism classes of simple modules with $p$-characters of height 0}\label{4}
\
\newline
\indent
In this section, we determine the isomorphism classes of simple modules of $\mathfrak{g}$ for $\lambda\in \mathbb{F}_p$.

\begin{pro}\label{pro3.3.1}
	Let $0\ne\lambda\in \mathbb{F}_p$, $0\ne \mu \in \mathbb{F}_p$. Then $K_\chi(\lambda)\cong K_\chi(\mu)$ if and only if $\lambda=\mu$.
\end{pro}
\begin{proof}
	Note that
	\begin{align*}
	K_\chi(\lambda)={\rm span}_\mathbb{F}\{D_1^{i_1}D_2^{i_2}D_3^j\otimes v_{\lambda,k,l}\mid 0\le i_1\le p^{t_1}-1, \,0\le i_2\le p^{t_2}-1,\,j,l=0,1,\,0\le k\le \lambda\,,k+l\le\lambda \}
	\end{align*}
	and
	\begin{align*}
	K_\chi(\mu)={\rm span}_\mathbb{F}\{D_1^{i_1}D_2^{i_2}D_3^j\otimes w_{\mu,k,l}\mid 0\le i_1\le p^{t_1}-1, \,0\le i_2\le p^{t_2}-1,\,j,l=0,1,\,0\le k\le  \mu\,,k+l\le\mu \},
	\end{align*}
	where $v_{\lambda,k,l}$ and $w_{\mu,k,l}$ are bases of the simple restricted $\mathfrak{g}_{[0]}$-modules $L^0(\lambda)$ and $L^0(\mu)$, respectively.
	Suppose $\varphi: K_\chi(\lambda)\longrightarrow K_\chi(\mu)$ is an isomorphism and
	\begin{align*}
	\varphi(1\otimes v_{\lambda,0,0})=&\sum_{i_1,i_2}a_{i_1,i_2,k}D_1^{i_1}D_2^{i_2}\otimes w_{\mu,k,0}+\sum_{i_1,i_2}b_{i_1,i_2,k}D_1^{i_1}D_2^{i_2}D_3\otimes w_{\mu,k,0}\\
	&+\sum_{i_1,i_2}c_{i_1,i_2,k}D_1^{i_1}D_2^{i_2}\otimes w_{\mu,k,1}+\sum_{i_1,i_2}d_{i_1,i_2,k}D_1^{i_1}D_2^{i_2}D_3\otimes w_{\mu,k,1},
	\end{align*}
	where $a_{i_1,i_2,k}=b_{i_1,i_2,k}=0$ for $k<0$ or $k>\mu$, $c_{i_1,i_2,k}=d_{i_1,i_2,k}=0$ for $k<0$ or $k\ge \mu$.
	
	Sufficiency is obvious, and we prove necessity below.
	Suppose that there exist $i_{1,0},i_{2,0}$ with $2\le i_{1,0}\le p^{t_1}-1$, $2\le i_{2,0}\le p^{t_2}-1$ such that $a_{i_{1,0},i_{2,0},k}\ne 0$, $b_{i_{1,0},i_{2,0},k}\ne 0$ for $0\le k\le \mu$ and $c_{i_{1,0},i_{2,0},k}\ne 0$, $d_{i_{1,0},i_{2,0},k}\ne 0$ for $0\le k\le\mu-1$.
	
	We have
	\begin{align*}
	0=&\phi((x_1^{(i_{1,0}-1)}x_2^{(i_{2,0}+1)}D_2-x_1^{(i_{1,0})}x_2^{(i_{2,0})}D_1)\cdot (1\otimes v_{\lambda,0,0}))\\
	=&(x_1^{(i_{1,0}-1)}x_2^{(i_{2,0}+1)}D_2-x_1^{(i_{1,0})}x_2^{(i_{2,0})}D_1)\cdot (\sum_{i_1,i_2}a_{i_1,i_2,k}D_1^{i_1}D_2^{i_2}\otimes w_{\mu,k,0}+\sum_{i_1,i_2}b_{i_1,i_2,k}D_1^{i_1}D_2^{i_2}D_3\otimes w_{\mu,k,0}\\
	&+\sum_{i_1,i_2}c_{i_1,i_2,k}D_1^{i_1}D_2^{i_2}\otimes w_{\mu,k,1}+\sum_{i_1,i_2}d_{i_1,i_2,k}D_1^{i_1}D_2^{i_2}D_3\otimes w_{\mu,k,1})\\
	 =&\sum_{i_1,i_2}(-1)^{l_1+l_2}a_{i_1,i_2,k}C_{i_1}^{l_1}C_{i_2}^{l_2}D_1^{i_1-l_1}D_2^{i_2-l_2}(x_1^{(i_{1,0}-l_1-1)}x_2^{(i_{2,0}-l_2+1)}D_2-x_1^{(i_{1,0}-l_1)}x_2^{(i_{2,0}-l_2)}D_1)\otimes w_{\mu,k,0}\\
	 &+\sum_{i_1,i_2}(-1)^{l_1+l_2}b_{i_1,i_2,k}C_{i_1}^{l_1}C_{i_2}^{l_2}D_1^{i_1-l_1}D_2^{i_2-l_2}D_3(x_1^{(i_{1,0}-l_1-1)}x_2^{(i_{2,0}-l_2+1)}D_2-x_1^{(i_{1,0}-l_1)}x_2^{(i_{2,0}-l_2)}D_1)\otimes w_{\mu,k,0}\\
	 &+\sum_{i_1,i_2}(-1)^{l_1+l_2}c_{i_1,i_2,k}C_{i_1}^{l_1}C_{i_2}^{l_2}D_1^{i_1-l_1}D_2^{i_2-l_2}(x_1^{(i_{1,0}-l_1-1)}x_2^{(i_{2,0}-l_2+1)}D_2-x_1^{(i_{1,0}-l_1)}x_2^{(i_{2,0}-l_2)}D_1)\otimes w_{\mu,k,1}\\
	 &+\sum_{i_1,i_2}(-1)^{l_1+l_2}d_{i_1,i_2,k}C_{i_1}^{l_1}C_{i_2}^{l_2}D_1^{i_1-l_1}D_2^{i_2-l_2}D_3(x_1^{(i_{1,0}-l_1-1)}x_2^{(i_{2,0}-l_2+1)}D_2-x_1^{(i_{1,0}-l_1)}x_2^{(i_{2,0}-l_2)}D_1)\otimes w_{\mu,k,1}\\
	=&a_{i_{1,0},i_{2,0},k}(-1)^{i_{1,0}+i_{2,0}-1}i_{2,0}D_2\otimes x_2D_1 \cdot w_{\mu,k,0}+b_{i_{1,0},i_{2,0},k}(-1)^{i_{1,0}+i_{2,0}-1}i_{2,0}D_2D_3\otimes x_2D_1 \cdot w_{\mu,k,0}\\
	&+c_{i_{1,0},i_{2,0},k}(-1)^{i_{1,0}+i_{2,0}-1}i_{2,0}D_2\otimes x_2D_1 \cdot w_{\mu,k,1}+d_{i_{1,0},i_{2,0},k}(-1)^{i_{1,0}+i_{2,0}-1}i_{2,0}D_2D_3\otimes x_2D_1 \cdot w_{\mu,k,1}\\
	&+\sum_{j_1\ge i_{1,0},j_2\ge i_{2,0},j_1+j_2\ne i_{1,0}+i_{2,0}}(-1)^{i_{1,0}+i_{2,0}-1}a_{i_1,i_2,k}C_{i_1}^{i_{1,0}}C_{i_2}^{i_{2,0}-1}D_1^{i_1-i_{1,0}}D_2^{i_2-i_{2,0}+1}\otimes x_2D_1 \cdot w_{\mu,k,0}\\
		&+\sum_{j_1\ge i_{1,0},j_2\ge i_{2,0},j_1+j_2\ne i_{1,0}+i_{2,0}}(-1)^{i_{1,0}+i_{2,0}-1}b_{i_1,i_2,k}C_{i_1}^{i_{1,0}}C_{i_2}^{i_{2,0}-1}D_1^{i_1-i_{1,0}}D_2^{i_2-i_{2,0}+1}D_3\otimes x_2D_1 \cdot w_{\mu,k,0}\\
			&+\sum_{j_1\ge i_{1,0},j_2\ge i_{2,0},j_1+j_2\ne i_{1,0}+i_{2,0}}(-1)^{i_{1,0}+i_{2,0}-1}c_{i_1,i_2,k}C_{i_1}^{i_{1,0}}C_{i_2}^{i_{2,0}-1}D_1^{i_1-i_{1,0}}D_2^{i_2-i_{2,0}+1}\otimes x_2D_1 \cdot w_{\mu,k,1}\\
				&+\sum_{j_1\ge i_{1,0},j_2\ge i_{2,0},j_1+j_2\ne i_{1,0}+i_{2,0}}(-1)^{i_{1,0}+i_{2,0}-1}d_{i_1,i_2,k}C_{i_1}^{i_{1,0}}C_{i_2}^{i_{2,0}-1}D_1^{i_1-i_{1,0}}D_2^{i_2-i_{2,0}+1}D_3\otimes x_2D_1 \cdot w_{\mu,k,1}\\
	&+\sum_{i_1,i_2}(-1)^{i_{1,0}+i_{2,0}-1}a_{i_1,i_2,k}C_{i_1}^{i_{1,0}-1}C_{i_2}^{i_{2,0}}D_1^{i_1-i_{1,0}+1}D_2^{i_2-i_{2,0}}\otimes(x_2D_2-x_1D_1)\cdot w_{\mu,k,0}\\
	&+\sum_{i_1,i_2}(-1)^{i_{1,0}+i_{2,0}-1}a_{i_1,i_2,k}C_{i_1}^{i_{1,0}-2}C_{i_2}^{i_{2,0}+1}D_1^{i_1-i_{1,0}+2}D_2^{i_2-i_{2,0}-1}\otimes x_1D_2\cdot w_{\mu,k,0}\\
	 &+\sum_{i_1,i_2}(-1)^{i_{1,0}+i_{2,0}}a_{i_1,i_2,k}(C_{i_1}^{i_{1,0}-1}C_{i_2}^{i_{2,0}+1}-C_{i_1}^{i_{1,0}}C_{i_2}^{i_{2,0}})D_1^{i_1-i_{1,0}+1}D_2^{i_2-i_{2,0}}\otimes  w_{\mu,k,0}\\
	 &+\sum_{i_1,i_2}(-1)^{i_{1,0}+i_{2,0}+1}b_{i_1,i_2,k}C_{i_1}^{i_{1,0}-1}C_{i_2}^{i_{2,0}}D_1^{i_1-i_{1,0}+1}D_2^{i_2-i_{2,0}}D_3\otimes(x_2D_2-x_1D_1)\cdot w_{\mu,k,0}\\
	&+\sum_{i_1,i_2}(-1)^{i_{1,0}+i_{2,0}}b_{i_1,i_2,k}C_{i_1}^{i_{1,0}-2}C_{i_2}^{i_{2,0}+1}D_1^{i_1-i_{1,0}+2}D_2^{i_2-i_{2,0}-1}D_3\otimes x_1D_2\cdot w_{\mu,k,0}\\
	 &+\sum_{i_1,i_2}(-1)^{i_{1,0}+i_{2,0}+1}b_{i_1,i_2,k}(C_{i_1}^{i_{1,0}-1}C_{i_2}^{i_{2,0}+1}-C_{i_1}^{i_{1,0}}C_{i_2}^{i_{2,0}})D_1^{i_1-i_{1,0}+1}D_2^{i_2-i_{2,0}}D_3\otimes  w_{\mu,k,0}\\
	&+\sum_{i_1,i_2}(-1)^{i_{1,0}+i_{2,0}+1}c_{i_1,i_2,k}C_{i_1}^{i_{1,0}-1}C_{i_2}^{i_{2,0}}D_1^{i_1-i_{1,0}+1}D_2^{i_2-i_{2,0}}\otimes(x_2D_2-x_1D_1)\cdot w_{\mu,k,1}\\
	&+\sum_{i_1,i_2}(-1)^{i_{1,0}+i_{2,0}}c_{i_1,i_2,k}C_{i_1}^{i_{1,0}-2}C_{i_2}^{i_{2,0}+1}D_1^{i_1-i_{1,0}+2}D_2^{i_2-i_{2,0}-1}\otimes x_1D_2\cdot w_{\mu,k,1}\\
	 &+\sum_{i_1,i_2}(-1)^{i_{1,0}+i_{2,0}+1}c_{i_1,i_2,k}(C_{i_1}^{i_{1,0}-1}C_{i_2}^{i_{2,0}+1}-C_{i_1}^{i_{1,0}}C_{i_2}^{i_{2,0}})D_1^{i_1-i_{1,0}+1}D_2^{i_2-i_{2,0}}\otimes  w_{\mu,k,1}\\
	 &+\sum_{i_1,i_2}(-1)^{i_{1,0}+i_{2,0}+1}d_{i_1,i_2,k}C_{i_1}^{i_{1,0}-1}C_{i_2}^{i_{2,0}}D_1^{i_1-i_{1,0}+1}D_2^{i_2-i_{2,0}}D_3\otimes(x_2D_2-x_1D_1)\cdot w_{\mu,k,1}\\
	&+\sum_{i_1,i_2}(-1)^{i_{1,0}+i_{2,0}}d_{i_1,i_2,k}C_{i_1}^{i_{1,0}-2}C_{i_2}^{i_{2,0}+1}D_1^{i_1-i_{1,0}+2}D_2^{i_2-i_{2,0}-1}D_3\otimes x_1D_2\cdot w_{\mu,k,1}\\
	 &+\sum_{i_1,i_2}(-1)^{i_{1,0}+i_{2,0}+1}d_{i_1,i_2,k}(C_{i_1}^{i_{1,0}-1}C_{i_2}^{i_{2,0}+1}-C_{i_1}^{i_{1,0}}C_{i_2}^{i_{2,0}})D_1^{i_1-i_{1,0}+1}D_2^{i_2-i_{2,0}}D_3\otimes  w_{\mu,k,1}.
	\end{align*}
	When $k\ne 0$, since $x_2D_1\cdot w_{\mu,k,0}=k(\mu+1-k)w_{\mu,k-1,0}$ and $x_2D_1\cdot w_{\mu,k,1}=k(\mu-k)w_{\mu,k-1,1}$, it is clear that $a_{i_{1,0},i_{2,0},k}= b_{i_{1,0},i_{2,0},k}=0$ for $1\le k\le \mu$ and $c_{i_{1,0},i_{2,0},k}=d_{i_{1,0},i_{2,0},k}=0$ for $1\le k\le \mu-1$. When $k=0$, from
	\begin{align*}
		0=&\phi((x_1^{(i_{1,0})}x_2^{(i_{2,0})}D_2-x_1^{(i_{1,0}+1)}x_2^{(i_{2,0}-1)}D_1)\cdot (1\otimes v_{\lambda,0,0}))\\
		=&(x_1^{(i_{1,0}-1)}x_2^{(i_{2,0}+1)}D_2-x_1^{(i_{1,0})}x_2^{(i_{2,0})}D_1)\cdot (\sum_{i_1,i_2}a_{i_1,i_2,0}D_1^{i_1}D_2^{i_2}\otimes w_{\mu,0,0}+\sum_{i_1,i_2}b_{i_1,i_2,0}D_1^{i_1}D_2^{i_2}D_3\otimes w_{\mu,0,0}\\
		&+\sum_{i_1,i_2}c_{i_1,i_2,0}D_1^{i_1}D_2^{i_2}\otimes w_{\mu,0,1}+\sum_{i_1,i_2}d_{i_1,i_2,0}D_1^{i_1}D_2^{i_2}D_3\otimes w_{\mu,0,1})
	\end{align*}
	and
	$x_1D_2\cdot w_{\mu,0,0}=w_{\mu,1,0}$ and $x_1D_2\cdot w_{\mu,0,1}=w_{\mu,1,1}$ that $a_{i_{1,0},i_{2,0},0}= b_{i_{1,0},i_{2,0},0}=0$ for $\mu\in\mathbb{F}_p$ and $c_{i_{1,0},i_{2,0},0}=d_{i_{1,0},i_{2,0},0}=0$ for $\mu\ne 1$. For $\mu=1$, similarly, we have
	\begin{align*}
			0=&\phi((x_1^{(i_{1,0}-1)}x_2^{(i_{2,0})}D_2-x_1^{(i_{1,0})}x_2^{(i_{2,0}-1)}D_1)\cdot (1\otimes v_{\lambda,0,0}))\\
		=&(x_1^{(i_{1,0}-1)}x_2^{(i_{2,0})}D_2-x_1^{(i_{1,0})}x_2^{(i_{2,0}-1)}D_1)\cdot (\sum_{i_1,i_2}a_{i_1,i_2,0}D_1^{i_1}D_2^{i_2}\otimes w_{\mu,0,0}+\sum_{i_1,i_2}b_{i_1,i_2,0}D_1^{i_1}D_2^{i_2}D_3\otimes w_{\mu,0,0}\\
		&+\sum_{i_1,i_2}c_{i_1,i_2,0}D_1^{i_1}D_2^{i_2}\otimes w_{\mu,0,1}+\sum_{i_1,i_2}d_{i_1,i_2,0}D_1^{i_1}D_2^{i_2}D_3\otimes w_{\mu,0,1}).
	\end{align*}
When $i_{1,0}\ne i_{2,0}$, based on the above equation, we can derive $$c_{i_{1,0},i_{2,0},0}(i_{1,0}-i_{2,0})D_1D_2\otimes w_{\mu,0,1}+d_{i_{1,0},i_{2,0},0}(i_{1,0}-i_{2,0})D_1D_2D_3\otimes w_{\mu,0,1}=0.$$ Therefore, we can get $c_{i_{1,0},i_{2,0},0}=d_{i_{1,0},i_{2,0},0}=0$.
When $i_{1,0}= i_{2,0}$, we have
	\begin{align*}
	0=&\phi((x_1^{(i_{1,0}-2)}x_2^{(i_{2,0})}D_2-x_1^{(i_{1,0}-1)}x_2^{(i_{2,0}-1)}D_1)\cdot (1\otimes v_{\lambda,0,0}))\\
	=&(x_1^{(i_{1,0}-2)}x_2^{(i_{2,0})}D_2-x_1^{(i_{1,0}-1)}x_2^{(i_{2,0}-1)}D_1)\cdot (\sum_{i_1,i_2}a_{i_1,i_2,0}D_1^{i_1}D_2^{i_2}\otimes w_{\mu,0,0}+\sum_{i_1,i_2}b_{i_1,i_2,0}D_1^{i_1}D_2^{i_2}D_3\otimes w_{\mu,0,0}\\
	&+\sum_{i_1,i_2}c_{i_1,i_2,0}D_1^{i_1}D_2^{i_2}\otimes w_{\mu,0,1}+\sum_{i_1,i_2}d_{i_1,i_2,0}D_1^{i_1}D_2^{i_2}D_3\otimes w_{\mu,0,1}).
\end{align*}
From the preceding equation, it follows that $$\frac{1}{2}c_{i_{1,0},i_{2,0},0}i_{1,0}(i_{1,0}+1)D_1^2D_2\otimes  w_{\mu,0,1}+d_{i_{1,0},i_{2,0},0}i_{1,0}(i_{1,0}+1)D_1^2D_2D_3\otimes  w_{\mu,0,1}=0.$$
Hence, $c_{i_{1,0},i_{2,0},0}=d_{i_{1,0},i_{2,0},0}=0$.
	
	In conclusion, $a_{i_{1,0},i_{2,0},k}= b_{i_{1,0},i_{2,0},k}=0$ for $0\le k\le \mu$ and $c_{i_{1,0},i_{2,0},k}=d_{i_{1,0},i_{2,0},k}=0$ for $0\le k\le \mu-1$, where $1\le i_{1,0}<p^{t_1}-1$, $1\le i_{2,0}<p^{t_2}-1$, which is a contradiction.

	Then by a direct analysis of the weight of $D_1^{i_1}D_2^{i_2}D_3^j\otimes w_{\mu,k,l}$ for $j=0,1$, we can assume that
	\begin{align*}
	\varphi(1\otimes v_{\lambda,0,0})=&a_{0,0,k}\otimes w_{\mu,k,0}+a_{p^{t_1}-1,p^{t_2}-1,k}D_1^{p^{t_1}-1}D_2^{p^{t_2}-1}\otimes w_{\mu,k,0}\\
	&+b_{0,0,k}D_3\otimes w_{\mu,k,0}+b_{p^{t_1}-1,p^{t_2}-1,k}D_1^{p^{t_1}-1}D_2^{p^{t_2}-1}D_3\otimes w_{\mu,k,0}\\
	&+c_{0,p^{t_2}-1,k}D_2^{p^{t_2}-1}\otimes w_{\mu,k,1}+d_{0,p^{t_2}-1,k}D_2^{p^{t_2}-1}D_3\otimes w_{\mu,k,1}
	\end{align*}
	or
	\begin{align*}
	\varphi(1\otimes v_{\lambda,0,0})=&a_{p^{t_1}-1,0,k}D_1^{p^{t_1}-1}\otimes w_{\mu,k,0}+b_{p^{t_1}-1,0,k}D_1^{p^{t_1}-1}D_3\otimes w_{\mu,k,0}\\
	&+c_{0,0,k}\otimes w_{\mu,k,1}+c_{p^{t_1}-1,p^{t_2}-1,k}D_1^{p^{t_1}-1}D_2^{p^{t_2}-1}\otimes w_{\mu,k,1}\\
	&+d_{0,0,k}D_3\otimes w_{\mu,k,1}+d_{p^{t_1}-1,p^{t_2}-1,k}D_1^{p^{t_1}-1}D_2^{p^{t_2}-1}D_3\otimes w_{\mu,k,1}
	\end{align*}
	or
	\begin{align*}
	\varphi(1\otimes v_{\lambda,0,0})=a_{0,p^{t_2}-1,k}D_2^{p^{t_2}-1}\otimes w_{\mu,k,0}+b_{0,p^{t_2}-1,k}D_2^{p^{t_2}-1}D_3\otimes w_{\mu,k,0}
	\end{align*}
	or
	\begin{align*}
	\varphi(1\otimes v_{\lambda,0,0})=c_{p^{t_1}-1,0,k}D_1^{p^{t_1}-1}\otimes w_{\mu,k,1}+d_{p^{t_1}-1,0,k}D_1^{p^{t_1}-1}D_3\otimes w_{\mu,k,1},
	\end{align*}
	where $a_{i_1,i_2,k}=b_{i_1,i_2,k}=0$ for $k<0$ or $k>\mu$, $c_{i_1,i_2,k}=d_{i_1,i_2,k}=0$ for $k<0$ or $k\ge \mu$, $i_1=0$ or $p^{t_1}-1$ and $i_2=0$ or $p^{t_2}-1$.
	
	If
	\begin{align*}
	\varphi(1\otimes v_{\lambda,0,0})=&a_{0,0,k}\otimes w_{\mu,k,0}+a_{p^{t_1}-1,p^{t_2}-1,k}D_1^{p^{t_1}-1}D_2^{p^{t_2}-1}\otimes w_{\mu,k,0}\\
	&+b_{0,0,k}D_3\otimes w_{\mu,k,0}+b_{p^{t_1}-1,p^{t_2}-1,k}D_1^{p^{t_1}-1}D_2^{p^{t_2}-1}D_3\otimes w_{\mu,k,0}\\
	&+c_{0,p^{t_2}-1,k}D_2^{p^{t_2}-1}\otimes w_{\mu,k,1}+d_{0,p^{t_2}-1,k}D_2^{p^{t_2}-1}D_3\otimes w_{\mu,k,1},
	\end{align*}
	since
	\begin{align*}
	0=&\varphi((\xi D_1+x_2D_3)\cdot (1\otimes v_{\lambda,0,0}))\\
	=&a_{0,0,k}k\otimes w_{\mu,k-1,1}+a_{p^{t_1}-1,p^{t_2}-1,k}(D_1^{p^{t_1}-1}D_2^{p^{t_2}-1}\otimes w_{\mu,k-1,1}-D_1^{p^{t_1}-1}D_2^{p^{t_2}-1}D_3\otimes w_{\mu,k,0})\\
	&+b_{0,0,k}(D_1\otimes w_{\mu,k,0}-kD_3\otimes w_{\mu,k-1,1})+b_{p^{t_1}-1,p^{t_2}-1,k}kD_1^{p^{t_1}-1}D_2^{p^{t_2}-1}D_3\otimes w_{\mu,k-1,1}\\
	&+c_{0,p^{t_2}-1,k}((\mu-k) D_2^{p^{t_2}-1}\otimes w_{\mu,k,0}-D_2^{p^{t_2}-2}D_3\otimes w_{\mu,k,1})+d_{0,p^{t_2}-1,k}D_2^{p^{t_2}-1}D_3\otimes w_{\mu,k,0},
	\end{align*}
	we can get $a_{p^{t_1}-1,p^{t_2}-1,k}=b_{0,0,k}=d_{0,p^{t_2}-1,k}=0$ for $0\le k\le \mu$ and
	$a_{0,0,k}=b_{p^{t_1}-1,p^{t_2}-1,k}=c_{0,p^{t_2}-1,k}=0$ for $1\le k\le\mu$. Hence,
	\begin{align*}
	\varphi(1\otimes v_{\lambda,0,0})=a_{0,0,0}\otimes w_{\mu,0,0}+b_{p^{t_1}-1,p^{t_2}-1,0}D_1^{p^{t_1}-1}D_2^{p^{t_2}-1}D_3\otimes w_{\mu,0,0}.
	\end{align*}
	 From comparing the weights of $1\otimes v_{\lambda,0,0}$, $1\otimes w_{\mu,0,0}$ and $D_1^{p^{t_1}-1}D_2^{p^{t_2}-1}D_3\otimes w_{\mu,0,0}$, we get $\lambda=\mu$.
	
	If
	\begin{align*}
	\varphi(1\otimes v_{\lambda,0,0})=&a_{p^{t_1}-1,0,k}D_1^{p^{t_1}-1}\otimes w_{\mu,k,0}+b_{p^{t_1}-1,0,k}D_1^{p^{t_1}-1}D_3\otimes w_{\mu,k,0}\\
	&+c_{0,0,k}\otimes w_{\mu,k,1}+c_{p^{t_1}-1,p^{t_2}-1,k}D_1^{p^{t_1}-1}D_2^{p^{t_2}-1}\otimes w_{\mu,k,1}\\
	&+d_{0,0,k}D_3\otimes w_{\mu,k,1}+d_{p^{t_1}-1,p^{t_2}-1,k}D_1^{p^{t_1}-1}D_2^{p^{t_2}-1}D_3\otimes w_{\mu,k,1},
	\end{align*}
we have
		\begin{align*}
	0=&\varphi((\xi D_1+x_2D_3)\cdot (1\otimes v_{\lambda,0,0}))\\
	=&a_{p^{t_1}-1,0,k}kD_1^{p^{t_1}-1}\otimes w_{\mu,k-1,1}+b_{p^{t_1}-1,0,k}(D_1^{p^{t_1}-1}D_3\otimes w_{\mu,k-1,0}-\chi(D_1)\otimes w_{\mu,k,0})\\
	&+c_{0,0,k}\otimes w_{\mu,k,0}+c_{p^{t_1}-1,p^{t_2}-1,k}((\mu-k) D_1^{p^{t_1}-1}D_2^{p^{t_2}-1}\otimes w_{\mu,k,0}-D_1^{p^{t_1}-1}D_2^{p^{t_2}-2}D_3\otimes w_{\mu,k,1})\\
	&+ d_{p^{t_1}-1,p^{t_2}-1,k}(\chi (D_1)D_2^{p^{t_2}-1}\otimes w_{\mu,k,1}+(\mu-k) D_1^{p^{t_1}-1}D_2^{p^{t_2}-1}D_3\otimes w_{\mu,k,0})\\
	&+d_{0,0,k}(D_1\otimes w_{\mu,k,0}+(\mu-k)D_3\otimes w_{\mu,k,0}).
	\end{align*}
	We discuss this in the following two cases.
	
		\textbf{Case 1\;:} $\chi(D_1)=1$.
		
		In this subcase, it is clear that $a_{p^{t_1}-1,0,k}=0$ for $1\le k\le \mu$, $b_{p^{t_1}-1,0,k}=0$ for $0\le k\le \mu$ and  $c_{0,0,k}=c_{p^{t_1}-1,p^{t_2}-1,k}=d_{0,0,k}=d_{p^{t_1}-1,p^{t_2}-1,k}=0$ for $0\le k\le \mu-1$. Therefore,
	\begin{align*}
\varphi(1\otimes v_{\lambda,0,0})=a_{p^{t_1}-1,0,0}D_1^{p^{t_1}-1}\otimes w_{\mu,0,0}.
\end{align*}
According to 
	\begin{align*}
	0=&\varphi(x_1^{(2)}D_2\cdot (1\otimes v_{\lambda,0,0}))=-a_{p^{t_1}-1,0,0}(2D_1^{p^{t_1}-3}D_2\otimes w_{\mu,0,0}+D_1^{p^{t_1}-2}\otimes  w_{\mu,1,0}),
	\end{align*}	
it follows that $a_{p^{t_1}-1,0,0}=0$, which is a contradiction.
	
	\textbf{Case 2\;:} $\chi(D_1)=0$.
	
	In this subcase, it is easy to see that $a_{p^{t_1}-1,0,k}=b_{p^{t_1}-1,0,k}=0$ for $1\le k\le \mu$ and $c_{0,0,k}=c_{p^{t_1}-1,p^{t_2}-1,k}=d_{0,0,k}=d_{p^{t_1}-1,p^{t_2}-1,k}=0$ for $0\le k\le \mu-1$. Accordingly,
	\begin{align*}
	\varphi(1\otimes v_{\lambda,0,0})=a_{p^{t_1}-1,0,0}D_1^{p^{t_1}-1}\otimes w_{\mu,0,0}+b_{p^{t_1}-1,0,0}D_1^{p^{t_1}-1}D_3\otimes w_{\mu,0,0}.
	\end{align*}
It follows from
	\begin{align*}
	0=\varphi(x_1^{(2)}D_2\cdot (1\otimes v_{\lambda,0,0}))=&-a_{p^{t_1}-1,0,0}(2D_1^{p^{t_1}-3}D_2\otimes w_{\mu,0,0}+D_1^{p^{t_1}-2}\otimes  w_{\mu,1,0})\\
	&-b_{p^{t_1}-1,0,0}(2D_1^{p^{t_1}-3}D_2D_3\otimes w_{\mu,0,0}+D_1^{p^{t_1}-2}D_3\otimes  w_{\mu,1,0})
	\end{align*}	
	that $a_{p^{t_1}-1,0,0}=b_{p^{t_1}-1,0,0}=0$ for $0\le k\le \mu$, which is a contradiction.
	
	If
	\begin{align*}
	\varphi(1\otimes v_{\lambda,0,0})=a_{0,p^{t_2}-1,k}D_2^{p^{t_2}-1}\otimes w_{\mu,k,0}+b_{0,p^{t_2}-1,k}D_2^{p^{t_2}-1}D_3\otimes w_{\mu,k,0},
	\end{align*}
	it follows from
	\begin{align*}
	0=&\varphi((\xi D_1+x_2D_3)\cdot (a_{0,p^{t_2}-1,k}D_2^{p^{t_2}-1}\otimes w_{\mu,k,0}+b_{0,p^{t_2}-1,k}D_2^{p^{t_2}-2}D_3\otimes w_{\mu,k,0})\\
	=&a_{0,p^{t_2}-1,k}(D_2^{p^{t_2}-1}D_3\otimes w_{\mu,k,0}+kD_2^{p^{t_2}-1}\otimes w_{\mu,k-1,0})\\
	&+b_{0,p^{t_2}-1,k}(D_1D_2^{p^{t_2}-1}\otimes w_{\mu,k,0}-kD_2^{p^{t_2}-1}D_3\otimes w_{\mu,k-1,0})
	\end{align*}
	that $a_{0,p^{t_2}-1,k}=b_{0,p^{t_2}-1,k}=0$, which is a contradiction.
	
	If
	\begin{align*}
	\varphi(1\otimes v_{\lambda,0,0})=c_{p^{t_1}-1,0,k}D_1^{p^{t_1}-1}\otimes w_{\mu,k,1}+d_{p^{t_1}-1,0,k}D_1^{p^{t_1}-1}D_3\otimes w_{\mu,k,1},
	\end{align*}
	from
	\begin{align*}
	0=&\varphi((\xi D_1+x_2D_3)\cdot (c_{p^{t_1}-1,0,k}D_1^{p^{t_1}-1}\otimes w_{\mu,k,1}+d_{p^{t_1}-1,0,k}D_1^{p^{t_1}-1}D_3\otimes w_{\mu,k,1})\\
	=&(\mu+k) c_{p^{t_1}-1,0,k}D_1^{p^{t_1}-1}\otimes w_{\mu,k,0}+d_{p^{t_1}-1,0,k}(1\otimes w_{\mu,k,1}-(\mu-k) D_1^{p^{t_1}-1}D_3\otimes w_{\mu,k,0})
	\end{align*}
	we can get $c_{p^{t_1}-1,0,k}=d_{p^{t_1}-1,0,k}=0$ for $0\le k\le \mu-1$, which is a contradiction.
	In summary, $\lambda=\mu$ for $0\ne \lambda,\mu\in \mathbb{F}_p$, as desired.
\end{proof}
\begin{pro}\label{pro3.10}
	Let $\lambda=0$, $ \mu \in \mathbb{F}_p$. Then $K_\chi(0)\cong K_\chi(\mu)$ if and only if $\mu=0$.
\end{pro}
\begin{proof}
	We can see that sufficiency is obvious, and below we prove necessity. We assume that $\mu\ne 0$. By discussing similarly with Proposition \ref{pro3.3.1} we can get $\varphi(1\otimes v_{0})=a_{0,0,0}\otimes w_{\mu,0,0}$. By comparing the weights of $1\otimes v_{0}$ and $\otimes w_{\mu,0,0}$, we get $\mu=0$, this contradicts our assumptions. Therefore, $\mu=0$, as desired.
\end{proof}

Based on Propositions \ref{pro3.5}-\ref{pro3.10} we can obtain the following main result.
\begin{thm}
	Let $\mathfrak{g}=H(2,1,\underline{t})$ be Hamiltonian superalgebras over $\mathbb{F}$, an algebraically closed field of prime characteristic $p>3$ and $t\in \mathbb{N}^{2}$.
	Let $\chi\in\mathfrak{g}^*$ with ${\rm ht}(\chi)=0$.
	For $\lambda \in \mathbb{F}_p$, generalized $\chi$-reduced
	Kac $\mathfrak{g}$-modules $K_\chi(\lambda)$ are simple.
	In addition, all non-isomorphic simple generalized
	$\chi$-reduced $\mathfrak{g}$-modules have been exhausted by $\{K_\chi(\lambda)\mid \lambda\in\mathbb{F}_p\}$.
	Below is a list of the dimensions of the simple modules:
	\begin{align*}
{\rm dim }K_\chi(\lambda)=\begin{cases}
	2p^{t_1+t_2} & \text{ if } \lambda=0 ,\\
	2(2\lambda+1)p^{t_1+t_2}& \text{ if } \lambda \ne 0.
	\end{cases}
	\end{align*}
\end{thm}

\end{document}